\pdfoutput=1
\RequirePackage{ifpdf}
\ifpdf 
\documentclass[pdftex]{sigma}
\else
\documentclass{sigma}
\fi

\newcommand{\om}{\omega}
\newcommand{\ka}{\kappa}
\newcommand{\si}{\sigma}
\newcommand{\Si}{\Sigma}

\newcommand{\un}{\underline}
\newcommand{\ba}{\mathcal{G}}

\newcommand{\va}{\varphi}

\newcommand{\fg}{\mathfrak g}
\newcommand{\fv}{\mathfrak q}
\newcommand{\fp}{\mathfrak p}

\newcommand{\fk}{\mathfrak k}

\newcommand{\fl}{\mathfrak l}

\newcommand{\fq}{\mathfrak q}

\newcommand{\Ad}{{\rm Ad}}
\newcommand{\Fl}{{\rm Fl}}

\newcommand{\id}{{\rm id}}

\newcommand{\na}{\nabla}
\newcommand{\U}{\Upsilon}

\newcommand{\Rho}{{\mbox{\sf P}}}
\newcommand{\R}{\mathbb{R}}

\newcommand{\ad}{{\rm ad}}

\newcommand{\C}{\mathbb{C}}

\numberwithin{equation}{section}

\newtheorem{Theorem}{Theorem}[section]
\newtheorem{Corollary}[Theorem]{Corollary}
\newtheorem{Lemma}[Theorem]{Lemma}
\newtheorem{Proposition}[Theorem]{Proposition}
 { \theoremstyle{definition}
\newtheorem{Definition}[Theorem]{Definition}
\newtheorem{Example}[Theorem]{Example}
\newtheorem{Remark}[Theorem]{Remark} }

\begin{document}


\newcommand{\arXivNumber}{1607.01965}

\renewcommand{\PaperNumber}{032}

\FirstPageHeading

\ShortArticleName{Local Generalized Symmetries and Locally Symmetric Parabolic Geometries}

\ArticleName{Local Generalized Symmetries\\ and Locally Symmetric Parabolic Geometries}

\Author{Jan GREGOROVI\v{C}~$^\dag$ and Lenka ZALABOV\'A~$^\ddag$}

\AuthorNameForHeading{J.~Gregorovi\v{c} and L.~Zalabov\'a}

\Address{$^\dag$~E.~\v{C}ech Institute, Mathematical Institute of Charles University,\\
\hphantom{$^\dag$}~Sokolovsk\'a 83, Praha 8 - Karl\'in, Czech Republic}
\EmailD{\href{mailto:jan.gregorovic@seznam.cz}{jan.gregorovic@seznam.cz}}

\Address{$^\ddag$~Institute of Mathematics and Biomathematics, Faculty of Science, University of South Bohemia\\
\hphantom{$^\ddag$}~in \v{C}esk\'e Bud\v{e}jovice, Brani\v{s}ovsk\'a 1760, \v{C}esk\'e Bud\v{e}jovice, 370 05, Czech Republic}
\EmailD{\href{mailto:lzalabova@gmail.com}{lzalabova@gmail.com}}

\ArticleDates{Received August 29, 2016, in f\/inal form May 18, 2017; Published online May 23, 2017}

\Abstract{We investigate (local) automorphisms of parabolic geometries that generalize geodesic symmetries. We show that many types of parabolic geometries admit at most one generalized geodesic symmetry at a point with non-zero harmonic curvature. Moreover, we show that if there is exactly one symmetry at each point, then the parabolic geometry is a~generalization of an af\/f\/ine (locally) symmetric space.}


\Classification{53C10; 53C22; 53C15; 53C05; 53B15; 53A55}

\section{Introduction}

Symmetric spaces are extremely useful geometric objects on smooth manifolds. There are also many generalizations of symmetric spaces appearing in several areas of dif\/ferential geometry and the theory of Lie groups and algebras. We are interested in generalizations of symmetric spaces in the setting of parabolic geometries, see \cite[Section~3.1]{parabook}. We consider regular normal parabolic geometries $(\ba\to M,\om)$ of type $(G,P)$ on smooth connected manifolds $M$. We assume that $G$ is a Lie group with a $|k|$-graded simple Lie algebra $\fg=\oplus_{i=-k}^{k}\fg_i$ and $P$ is the parabolic subgroup of $G$ with the Lie algebra $\fp=\oplus_{i=0}^{k}\fg_i$ such that the Klein geometry $(G,P)$ is ef\/fective. We f\/ix the reductive Levi decomposition $P=G_0\rtimes \exp (\fp_+)$, where $\fp_+:=\oplus_{i=1}^{k}\fg_i$ and $G_0$ is the Lie group of grading preserving elements of $P$. We write $\fg_-:=\oplus_{i=-k}^{-1}\fg_i$.

Regular normal parabolic geometries provide a solution to the equivalence problem for a~wide class of geometric structures. In the f\/irst step, so called prolongation, one constructs the $P$-bundle $\ba$ over $M$ and the Cartan connection $\om$, which is a $P$-equivariant $\fg$-valued absolute parallelism on $\ba$ that reproduces the generators of fundamental vector f\/ields of the $P$-action. The precise process of the prolongation is not directly related to the results presented in this article and will not be reviewed. In the second step, one computes the harmonic curvature $\kappa_H$ which is the basic invariant of all normal parabolic geometries that (in principle) solves the equivalence problem for normal parabolic geometries. We recall that $\kappa_H$ is the projection of the curvature $[\om,\om]+d\om$ of the Cartan connection $\om$ viewed as a function $\kappa\colon \ba\to \wedge^2 (\fg/\fp)^*\otimes \fg$ into the cohomology space $H^2(\fg_-,\fg)$ of cochains on $\fg_-$ with values in $\fg$.

A (local) automorphism of $(\ba\to M,\om)$ is a (local) $P$-bundle morphism $\va$ on $\ba$ such that $\va^*\om=\om$ holds. We denote by $\un{\va}$ the underlying (local) dif\/feomorphism of $\va$ on $M$. We say that a~(local) dif\/feomorphism $f$ on $M$ preserves the parabolic geometry $(\ba\to M,\om)$ if $f=\un{\va}$ for some (local) automorphism $\va$ of $(\ba\to M,\om)$. Local automorphisms of parabolic geometries are uniquely determined by the underlying dif\/feomorphisms under our assumption of ef\/fectivity of the Klein geometry~$(G,P)$. We are interested in a class of (local) dif\/feomorphisms $f$ on $M$ for which we know a priori the (local) $P$-bundle morphisms $\va$ on $\ba$ covering (local) dif\/feomor\-phisms~$f$ and we ask when they preserve the parabolic geometry, see Def\/inition~\ref{def3}. Let us explain that these dif\/feomorphisms are closely related to geodesic symmetries.

We recall that a normal coordinate system of a linear connection $\nabla$ on $M$ given by the frame~$u$ of~$T_xM$ is given by projections of f\/lows $\Fl_t^{B(X)}$ of the standard horizontal vector f\/ields~$B(X)$ for $X\in \mathbb{R}^n$ on the f\/irst-order frame bundle starting at~$u$, see \cite[Section~III.8]{KoNo}. Indeed, the projection of $\Fl_t^{B(X)}(u)$ onto $M$ is the geodesic of $\nabla$ going through $x$ in the direction with coordina\-tes~$X$ in the frame~$u$. A geodesic symmetry of $\nabla$ at the point $x$ is the unique dif\/feomorphism with coordinates $-\id_{\mathbb{R}^n}$ in all normal coordinate system given by any frame $u$ of $T_xM$.

The pair $(M,\nabla)$ is an af\/f\/ine locally symmetric space if each geodesic symmetry of $\nabla$ is an af\/f\/ine transformation.
In \cite{Helgason} or \cite{Besse} the authors studied the theory of symmetric spaces, where the geodesic symmetries preserve a geometric structure such as Riemannian metric or quaternionic K\"ahler structure. The f\/irst author classif\/ied in \cite{G3} all parabolic geometries preserved by all geodesic symmetries on semisimple symmetric spaces. Typical examples of such parabolic geometries are provided by the projective class of $\nabla$ of the af\/f\/ine (locally) symmetric space $(M,\nabla)$ or the conformal class of the metric on the Riemannian symmetric space or the (para)-quaternionic geometry given by the (para)-quaternionic K\"ahler symmetric space.

A normal coordinate system on the parabolic geometry $(p\colon \ba\to M,\om)$ given by $u\in \ba$ is given by projections $p\circ \Fl_t^{\om^{-1}(X)}(u)$ of f\/lows of the constant vector f\/ields $\om^{-1}(X)$ for coordinates $X\in \fg_-$. If we consider (local) dif\/feomorphisms $f$ on $M$ that are linear in some normal coordinate system of $(\ba\to M,\om)$, then we know a priori the (local) $P$-bundle morphisms $\va$ on $\ba$ covering (local) dif\/feomorphisms~$f$ and we ask when they preserve the parabolic geometry. The action of $G_0$ on $\ba$ induces a linear change of the normal coordinates, but the change of coordinates induced by the action of $\exp(\fp_+)$ is highly non-linear. Nevertheless, we can consider the class of (local) automorphisms of parabolic geometries with the property that their underlying (local) dif\/feomorphisms on $M$, analogously to geodesic symmetries, have the same coordinates in all normal coordinate systems in which the coordinates are linear.

\begin{Definition} \label{def3}
For $s$ in the center $Z(G_0)$ of $G_0$ and $u\in \ba$, let $s_u$ be the (local) $P$-bundle morphism of $\ba$ induced by the formula
\begin{gather*}s_u(\Fl_1^{\om^{-1}(X)}(u)):=\Fl_1^{\om^{-1}(X)}(us)=\Fl_1^{\om^{-1}(\Ad(s) (X))}(u)s\end{gather*}
for all $X$ in a maximal possible neighbourhood of $0$ in $\fg_-$ preserved by $\Ad(s)$.
\begin{enumerate}\itemsep=0pt
\item The (local) $P$-bundle morphism $s_u$ is a \emph{$($local$)$ $s$-symmetry} of the parabolic geometry $(\ba\to M,\om)$ at $x=p(u)$ if $s_u^*\om=\om$.
\item We write $\un s_u$ for the underlying (local) dif\/feomorphism on $M$ of the $P$-bundle mor\-phism~$s_u$ which has coordinates $\Ad(s)\in {\rm Gl}(\fg_-)$ in the normal coordinate system given by~$u$.
\item All (local) $s$-symmetries at all $x\in M$ for all $s\in Z(G_0)$ together are called \emph{$($local$)$ gene\-ra\-li\-zed symmetries} of parabolic geometries.
\item The parabolic geometry is \emph{$($locally$)$ $s$-symmetric} if there is a (local) $s$-symmetry at each point of $M$.
\end{enumerate}
\end{Definition}
\begin{Remark}
We always assume that $s$ is not the identity element $e$ in $Z(G_0)$, because $\id_\ba$ is the unique $e$-symmetry of each parabolic geometry, and therefore, the results presented in this article are trivial for $e$-symmetries.
\end{Remark}

Firstly, let us focus on the automorphisms $\varphi$ of parabolic geometries such that $\un{\varphi}$ has coordinates $-\id_{\fg_-}$ in the normal coordinate system given by $u\in \ba$. If such an automorphism exists, then there is $m\in Z(G_0)$ such that $\Ad(m)=-\id_{\fg_-}$. Therefore we will always speak about (local) $m$-symmetries in this case.

The bundle morphisms $m_u$ (and thus dif\/feomorphisms $\underline{m}_u$) are generally dif\/ferent for dif\/fe\-rent~$u$ from the f\/iber~$\ba_x$ over~$x$ and each of them can be a (local) $m$-symmetry. In particular, there can be inf\/initely many (local) $m$-symmetries at $x$. In fact, this is the case of all mo\-dels~$G/P$ of AHS-structures, where the bundle maps $m_u$ are $m$-symmetries for all $u\in \ba$. On the other hand, the second author proved in \cite[Theorem 2.5]{LZ1} that projective, conformal and (para)-quaternionic geometries are the only types of parabolic geometries allowing $m$-symmetries at a~point~$x$ with a~non-zero Weyl (harmonic) curvature. Moreover, there is at most one $m$-symmetry at the point~$x$ with a non-zero Weyl curvature.

The second author showed in \cite[Theorem~3.2]{LZ1} that if a geodesic symmetry at $x$ for some linear connection on~$M$ is an automorphism of $(\ba\to M,\om)$, then the geodesic symmetry has coordinates $-\id_{\fg_-}$ in the normal coordinate system given by some $u\in \ba_x$. We prove in this article that there is the following characterization of non-f\/lat parabolic geometries which are preserved by all geodesic symmetries on af\/f\/ine (locally) symmetric spaces.

\begin{Theorem}\label{t1.1}
Suppose there is a parabolic geometry on a smooth connected manifold $M$ with a non-zero harmonic curvature at one point. Then the following claims are equivalent:
\begin{enumerate}\itemsep=0pt
\item[$1.$] The parabolic geometry is (locally) $m$-symmetric, i.e., at each point $x$ of $M$, there is a~$($local$)$ automorphism of the parabolic geometry such that the underlying $($local$)$ diffeomorphism on~$M$ has coordinates $-\id_{\fg_-}$ in the normal coordinate system for some $u\in \ba_x$.
\item[$2.$] The parabolic geometry is preserved by each geodesic symmetry on an affine $($locally$)$ symmetric space $(M,\nabla)$.
\end{enumerate}

In particular, if one of the above claims is satisfied, then the parabolic geometry is $($locally$)$ homogeneous, the affine $($locally$)$ symmetric space $(M,\nabla)$ from Claim~$(2)$ is unique and $\nabla$ is a~distinguished $($Weyl$)$ connection of the parabolic geometry.
\end{Theorem}
\begin{Remark}
Let us emphasize that (local) $m$-symmetries can appear only on $|1|$-graded parabolic geometries and only the projective, conformal and (para)-quaternionic geometries (and their complexif\/ications) can satisfy the assumptions and conditions of Theorem~\ref{t1.1}.
\end{Remark}

The global version of this statement was proved in \cite{Podesta} for projective geometries and in \cite[Corollary 4.5]{LZ3} for conformal and (para)-quaternionic geometries under the additional assumption of homogeneity or under the assumption that $m$-symmetries depend smoothly on the point~$x$. In \cite[Theorem~1]{GZ4}, we proved the global version of Theorem~\ref{t1.1} for conformal geometries. In this article, we obtain Theorem~\ref{t1.1} as a special case of Theorem~\ref{t1.3}.

There are many other interesting types of parabolic geometries, e.g., parabolic contact geometries, where there is no element $m\in P$ such that $\Ad(m)=-\id_{\fg_-}$. Thus they cannot be preserved by geodesic symmetries of any af\/f\/ine (locally) symmetric space. On the other hand, there are generalizations of symmetric spaces appearing in the literature that are nearly related to contact geometries.
In \cite{BFG} and \cite{KZ} the authors study sub-Riemannian and CR geometries preserved by so-called geodesic ref\/lexions on ref\/lexion spaces, see \cite{Loos}. A geodesic ref\/lexion on a ref\/lexion space is given by an endomorphism $s\in {\rm Gl}(\mathbb{R}^n)$ such that $s^2=\id_{\mathbb{R}^n}$ in a normal coordinate system of an admissible linear connection on the ref\/lexion space, see \cite{Loos}.

We studied in \cite{disertace,GZ} parabolic geometries on ref\/lexion spaces preserved by geodesic ref\/lexions. We proved that a geodesic ref\/lexion at $x$ preserving a parabolic geometry $(\ba\to M,\om)$ is given by an endomorphism $\Ad(s)\in {\rm Gl}(\fg_-)$ for some $s\in G_0$ such that $s^2=\id$ in a normal coordinate system of the parabolic geometry given by some $u\in \ba_x$. However, if $s\in G_0$ is not contained in $Z(G_0)$, then we cannot expect the uniqueness of the automorphisms $\varphi$ such that $\un{\varphi}$ has coordinates $\Ad(s)\in {\rm Gl}(\fg_-)$ in the normal coordinate system given by some $u\in \ba_x$. Indeed, if there is an other automorphism $\psi$ such that $\un{\psi}$ has coordinates $\Ad(g_0)\in {\rm Gl}(\fg_-)$ for some $g_0\in G_0$ in the normal coordinate system given by $u\in \ba_x$, then $\psi\varphi\psi^{-1}$ is in general a dif\/ferent automorphism such that $\un{\psi\varphi\psi^{-1}}$ has coordinates $\Ad(s)\in {\rm Gl}(\fg_-)$ in the normal coordinate system given by $ug_0\in \ba_x$. On the other hand, the second author proved in \cite[Section~5]{LZ2} that on some parabolic contact geometries, there is at most one $s$-symmetry at a~point $x$ with a~non-zero harmonic curvature for $s\in Z(G_0)$ such that $\Ad(s)|_{\fg_{-1}}=-\id$. We prove in this article that this holds for all parabolic contact geometries.

We classif\/ied in \cite{GZ3} all elements $s\in Z(G_0)$ that can appear as coordinates of underlying dif\/feomorphisms of automorphisms of parabolic geometries in a normal coordinate system at a~point with a~non-zero harmonic curvature. For example, we have found out that for complex $|1|$-graded parabolic geometries with a harmonic curvature of homogeneity~$3$, we have to consider elements $s\in Z(G_0)$ such that $s^3=\id$. Moreover, we constructed in \cite[Proposition~6.1]{GZ2} and \cite[Proposition~7.2]{GZ3} examples of such parabolic geometries on $\mathbb{Z}_3$-symmetric spaces, which are generalizations of symmetric spaces that are studied in~\cite{Kowalski}.

In fact, there are many known examples of (locally) $s$-symmetric parabolic geometries. Each locally f\/lat parabolic geometry is locally $s$-symmetric for each $s\in Z(G_0)$. We classif\/ied in \cite{GZ3} the elements $s\in Z(G_0)$ for which all locally $s$-symmetric parabolic geometries are f\/lat. Further, we showed in \cite[Proposition 6.1]{GZ2} that all submaximally symmetric parabolic geometries constructed in \cite[Section 4.1]{KT} are locally $s$-symmetric parabolic geometries for elements $s\in Z(G_0)$ that do not impose f\/latness. Let us emphasize that some of these examples carry more than one $s$-symmetry at each point and explicit examples can be found in \cite[Section 6]{GZ2}.
This shows that the results we obtain in this article do not hold for all types of parabolic geometries.
There are also further examples of (locally) $s$-symmetric parabolic geometries in \cite{BFG,G2,G3,GZ4,Podesta}.

Let us now summarize our main results for (local) $s$-symmetries and (locally) $s$-symmetric parabolic geometries we obtain in this article. The f\/irst main result states that there is a large class of types of parabolic geometries whose algebraic structure enforces uniqueness of (local) $s$-symmetries at points with a non-zero harmonic curvature. We characterize these types in a~way that is related to the theory of prolongations of annihilators of the harmonic curvature and the prolongation rigidity from \cite[Section~3.4]{KT} as follows.

\begin{Definition}\label{prolong} \label{prolongation-rigid}
Let $\mu$ be a component of the harmonic curvature (irreducible as a $G_0$-submodule of $H^2(\fg_-,\fg)$) of regular normal parabolic geometries of type $(G,P)$.
\begin{enumerate}\itemsep=0pt
\item For $\phi\in \mu$, let us denote by
\begin{gather*}\operatorname{ann}(\phi):=\{A\in \fg_0\colon A. \phi=0\}\end{gather*}
the \emph{annihilator} of $\phi$ in $\fg_{0}$. We def\/ine the \emph{$i$th prolongation of the annihilator} of $\phi$ as
\begin{gather*}\operatorname{pr}(\phi)_i=\{Z\in \fg_{i}\colon \ad(X_1)\cdots \ad(X_i)(Z)\in \operatorname{ann}(\phi) \ \rm{ for\ all\ } X_1,\dots, X_i\in\fg_{-1}\}.\end{gather*}
\item For $s\in Z(G_0)$, we say that the triple $(\fg,\fp,\mu)$ is \emph{prolongation rigid outside of the $1$-eigen\-space of $s$} if for all weights $\phi\in \mu$, all prolongations of the annihilator of $\phi$ in $\fg_0$ are contained in the $1$-eigenspace of $s$.
\end{enumerate}
\end{Definition}

We will see that there are triples $(\fg,\fp,\mu)$ that are prolongation rigid outside of the $1$-eigenspace of $s$ only for some $s\in Z(G_0)$. In particular, such triples $(\fg,\fp,\mu)$ are not prolongation rigid. Indeed, a triple $(\fg,\fp,\mu)$ is prolongation rigid if and only if it is prolongation rigid outside of the $1$-eigenspace of~$s$ for all $s\in Z(G_0)$.

In Section \ref{sec4.2}, we show how to classify all triples $(\fg,\fp,\mu)$ that are prolongation rigid outside of the $1$-eigenspace of $s$ for some $s\in Z(G_0)$ using the results in \cite{KT}. The following Theorem shows that for our purposes, it is enough to carry out the classif\/ication only for the compo\-nents~$\mu$ that are contained in the $1$-eigenspace of $s$.

\begin{Theorem}\label{t1.2}
Consider a triple $(\fg,\fp,\mu)$.

If $s\in Z(G_0)$ is such that $\mu$ is not contained in the $1$-eigenspace of $s$, then there is no $($local$)$ $s$-symmetry of a parabolic geometry of type $(G,P)$ at each point $x$ with a non-zero component of the harmonic curvature in~$\mu$.

If $s\in Z(G_0)$ is such that $(\fg,\fp,\mu)$ is prolongation rigid outside of the $1$-eigenspace of $s$, then there is at most one $($local$)$ $s$-symmetry of a parabolic geometry of type $(G,P)$ at the point~$x$ with a non-zero component of the harmonic curvature in~$\mu$.
\end{Theorem}

We proved Theorem \ref{t1.2} in \cite[Theorem 1.3]{GZ3} under the assumption that the parabolic geometry is homogeneous, but we can also easily construct non-homogeneous (locally) $s$-symmetric parabolic geometries of type $(G,Q)$ for certain triples $(\fg,\fq,\mu)$ that are prolongation rigid outside of the $1$-eigenspace of $s$. It suf\/f\/ices to consider correspondence spaces for parabolic subgroups $Q\subset P\subset G$ over (locally) $s$-symmetric parabolic geometries of type $(G,P)$ for $(\fg,\fp,\mu)$ that is prolongation rigid outside of the $1$-eigenspace of $s$, see \cite[Proposition 6.1]{GZ3}. We prove Theorem~\ref{t1.2} in Section~\ref{sec3.1}.

Let us now focus on (locally) $s$-symmetric parabolic geometries. We say that a map $S$ that picks a (local) $s$-symmetry at each point of~$M$ is a system of (local) $s$-symmetries. In general, systems of (local) $s$-symmetries are neither smooth nor unique. The conditions in Theorem~\ref{t1.2} can be used to prove the uniqueness of a system of (local) $s$-symmetries.

Our second main result concerns the conditions for the smoothness of a system of (local) $s$-symmetries. We consider the following generalization of af\/f\/ine locally symmetric spaces. There is a class of Weyl connections on each parabolic geometry playing a signif\/icant role in the theory of parabolic geometries, see \cite[Chapter 5]{parabook} and Section \ref{sec2.1}. Each Weyl connection is given by a reduction of~$\ba$ to $G_0$, i.e., by a smooth $G_0$-equivariant section $\si$ of the projection from~$\ba$ to~$\ba_0:=\ba/\exp(\fp_+)$. The sections $\si$ are called Weyl structures and we denote by $\nabla^\si$ the Weyl connection given by the Weyl structure~$\si$. Each point of $\si(\ba_0)_x$ def\/ines a dif\/ferent frame of~$T_xM$. However, the (local) dif\/feomorphism with coordinates $\Ad(s)\in {\rm Gl}(\fg_-)$ in a normal coordinate system of a Weyl connection~$\nabla^\si$ given by a~frame $\si(u_0)\in \si(\ba_0)_x$ is independent of the actual choice of $u_0\in (\ba_0)_x$, see Section~\ref{sec4.2}. We denote such a (local) dif\/feomorphism by $s_x^\si$.

If we choose a class of Weyl connections satisfying $T_xs_x^\si=T_xs_x^{\si'}$ for all Weyl connec\-tions~$\nabla^\si$,~$\nabla^{\si'}$ in the class and all $x\in M$, then the tangent bundle~$TM$ has a common decomposition into smooth subbundles according to the eigenvalues of $T_xs_x^\si$ for all $\nabla^\si$ in the class. We can further consider a subclass $[\nabla]$ of such a class of Weyl connections that restrict to the same partial connection on all smooth subbundles of $TM$ for all eigenvalues of $T_xs_x^\si$ dif\/ferent from $1$. Such a subclass $[\nabla]$ is equivalently characterized by the condition that the $1$-forms~$\U$ measuring the `dif\/ferences' (see the formula~(\ref{change-wc})) between arbitrary two connections in~$[\nabla]$ satisfy $(s_x^\si)^*\U(x)=\U(x)$ for all $x\in M$ and some (and thus all) Weyl connections $\nabla^\si\in [\nabla]$. In general, (local) dif\/feomorphisms $s_x^\si$ are dif\/ferent for dif\/ferent Weyl connections $\nabla^\si\in [\nabla]$. Therefore we can consider smooth maps $\un S$ assigning some of these dif\/feomorphisms to each $x\in M$. Equivalently we can directly assign to each $x\in M$ the Weyl structure $\si$ def\/ining $s_x^\si$.

\begin{Definition} \label{almost-invariant-ws}
Let $[\nabla]$ be a maximal subclass of the class of Weyl connections satisfying that
\begin{itemize}\itemsep=0pt
\item $T_xs_x^\si=T_xs_x^{\si'}$ holds for all Weyl connections $\nabla^\si,\nabla^{\si'}\in [\nabla]$ and all $x\in M$,
\item all connections in $[\nabla]$ restrict to the same partial connection on all smooth subbundles of~$TM$ for all eigenvalues of $T_xs_x^\si$ dif\/ferent from~$1$.
\end{itemize}
Let $\un S$ be a smooth map assigning to each $x\in M$ the (local) dif\/feomorphism $s_x^\si$ for some Weyl connection~$\nabla^\si$ (depending on~$x$) in~$[\nabla]$.
\begin{enumerate}\itemsep=0pt
\item The class $[\nabla]$ is called \emph{$\un S$-invariant} if $\un S(x)^*\nabla^\si\in [\nabla]$ and $\un S(x)^*\nabla^\si(x)=\nabla^\si(x)$ hold for some (and thus each) Weyl connection $\nabla^\si\in [\nabla]$ and all $x\in M$.
\item Weyl connections $\nabla$ in the $\un S$-invariant class $[\nabla]$ are called \emph{almost $\un S$-invariant Weyl connection.}
\item The almost $\un S$-invariant Weyl connection $\nabla$ is called \emph{invariant at} $x\in M$ if $\un S(x)^*\nabla=\nabla$.
\item The almost $\un S$-invariant Weyl connection $\nabla$ is called \emph{$\un S$-invariant} if $\un S(x)^*\nabla=\nabla$ holds for all $x\in M$.
 \end{enumerate}
\end{Definition}

We show in Section \ref{sec4.1} that if there is an almost $\un S$-invariant Weyl connection, then each~$\un S(x)$ preserves $(\ba\to M,\om)$, i.e., $\un S$ def\/ines a smooth system $S$ of (local) $s$-symmetries such that~$\un{S}(x)$ are the underlying (local) dif\/feomorphisms of $S(x)$ for all $x$. Thus the notation $\un S$ is consistent with Def\/inition~\ref{def3}.

If there is a smooth system $S$ of (local) $s$-symmetries of $(\ba\to M,\om)$, then we need the prolongation rigidity outside of the $1$-eigenspace of $s$ to show the existence of an $\un S$-invariant class of Weyl connections, see Section~\ref{sec4.3}. For all $|1|$-graded parabolic geometries and $s$ such that $\Ad(s)=-\id_{\fg_-}$, i.e., $s=m$, we obtain af\/f\/ine (locally) symmetric spaces, because the class $[\nabla]$ consists of a single connection. For all parabolic contact geometries and $s$ such that $\Ad(s)|_{\fg_{-1}}=-\id$ we obtain ref\/lexion spaces, but the $\un S$-invariant class $[\nabla]$ is not the class of admissible connections from \cite{Loos}.

For triples $(\fg,\fp,\mu)$ that are prolongation rigid outside of the $1$-eigenspace of $s$, we get the following existence result, which in particular implies Theorem \ref{t1.1}.

\begin{Theorem}\label{t1.3}
Suppose $s\in Z(G_0)$ is such that $(\fg,\fp,\mu)$ is prolongation rigid outside of the $1$-eigenspace of $s$. Suppose that the parabolic geometry $(\ba\to M,\om)$ of type $(G,P)$ has everywhere non-zero component of the harmonic curvature in~$\mu$. Then the following conditions are equivalent:
\begin{enumerate}\itemsep=0pt
\item[$1.$] The parabolic geometry is $($locally$)$ $s$-symmetric.
\item[$2.$] There is a smooth system $S$ of $($local$)$ $s$-symmetries.
\item[$3.$] There is an $\un S$-invariant class $[\nabla]$ of Weyl connections.
\end{enumerate}
Moreover, the smooth system $S$ is unique and $\un S$ consists of the underlying diffeomorphisms of~$S$ on~$M$. The equality $S(x) \circ S(y)=S(\un S(x)(y)) \circ S(x)$ holds whenever the compositions are defined. If $\Ad(s)\in {\rm Gl}(\fg_-)$ has no eigenvalue $1$, then $[\nabla]$ consists of a single $\un S$-invariant Weyl connection, which is locally affinely homogeneous.
\end{Theorem}

We prove the claims of Theorem \ref{t1.3} except the last one in Section~\ref{sec3}. The last claim does not hold without additional assumptions on the $1$-eigenspace. We prove the last claim in Section~\ref{sec6}, where we study additional properties that follow from assumptions on the position and shape of the $1$-eigenspace of $s$ in~$\fg_-$.

{\bf Outline of the article.} We recall basic facts and formulas for Weyl connections in Section~\ref{sec2}. In particular, we characterize automorphisms of parabolic geometries with their actions on Weyl structures. We recall the relation between normal coordinates and normal Weyl structures.

In Section~\ref{SEC3}, we prove Theorem~\ref{t1.2} and we provide the characterization of the triples $(\fg,\fp,\mu)$ that are prolongation rigid outside of the $1$-eigenspace of~$s$.

In Section \ref{sec3}, we prove Theorem~\ref{t1.3}. We also obtain further properties of (locally) $s$-symmetric parabolic geometries of type $(G,P)$ that have everywhere non-zero component of the harmonic curvature in $\mu$ for the triples $(\fg,\fp,\mu)$ that are prolongation rigid outside of the $1$-eigenspace of~$s$.

In Section \ref{sec6}, we classify all triples $(\fg,\fp,\mu)$ that are prolongation rigid outside of the $1$-eigenspace of $s$ such that $\mu$ is in the $1$-eigenspace of $s$. The classif\/ication is separated in the tables according to the common properties of the triples $(\fg,\fp,\mu)$ and elements $s\in Z(G_0)$. The notation for the tables and details on the classif\/ication can be found in Section~\ref{sec61}. We show in Section~\ref{sec5.1} that there are triples $(\fg,\fp,\mu)$ for which the $\un S$-invariant class $[\nabla]$ of Weyl connections consists of a single $\un S$-invariant Weyl connection. In particular, such an $\un S$-invariant Weyl connection is always (locally) af\/f\/inely homogeneous. In Sections~\ref{sec5.2} and~\ref{sec5.4}, we show that there are triples $(\fg,\fp,\mu)$ for which the (locally) $s$-symmetric parabolic geometries are locally correspondence spaces over some other $s$-symmetric parabolic geometries. In Section \ref{sec5.3}, we prove that there are triples $(\fg,\fp,\mu)$ for which the condition of (local) homogeneity is satisf\/ied for more complicated $\un S$-invariant class~$[\nabla]$ of Weyl connections.

In the Appendix \ref{sec5} we recall from~\cite{GZ2} the construction of (locally) homogeneous $s$-symmetric parabolic geometries that we need in Section~\ref{sec6}.

\section{Automorphisms of parabolic geometries} \label{sec2}
In this Section, we introduce necessary techniques and establish notation from the theory of parabolic geometries that we will use in the article, see \cite[Section 5.1]{parabook}. We focus here on actions of automorphisms on Weyl structures and connections.

\subsection{Weyl structures and connections}\label{sec2.1}

Consider a parabolic geometry $(\ba \to M, \om)$ of type $(G,P)$. Many geometric objects on $M$ can be identif\/ied with sections of natural bundles $\mathcal{V}$ associated to the $P$-bundle $\ba$ for representa\-tions~$V$ of~$P$. We can equivalently view the sections of~$\mathcal{V}$ as $P$-equivariant functions $\ba \to V$. In other words, the points of $\ba$ are (higher-order) frames and the $P$-equivariant functions are the coordinate functions. A~crucial tool that allows us to reduce the number and order of the frames are Weyl structures. A~\emph{$($local$)$ Weyl structure} is a (local) $G_0$-equivariant section $\si\colon \ba_0 \to \ba$ of the projection $\pi\colon \ba \to \ba_0$, where $\ba_0:=\ba /\exp(\fp_+)$ and $p_0\colon \ba_0\to M$ is a $G_0$-bundle over~$M$.

\begin{Definition}
Assume $\si\colon \ba_0\to \ba$ is a Weyl structure. Then for a section $\tau$ of a~natural bundle $\mathcal{V}$, we denote by $(\tau)_\si$ the $G_0$-equivariant function $\ba_0 \to V$ satisfying \begin{gather*}(\tau)_\si:=t\circ\si,\end{gather*} where $t\colon \ba \to V$ is the $P$-equivariant function corresponding to $\tau$.
\end{Definition}

In particular, vector f\/ields $\xi$ and $1$-forms $\U$ on $M$ are sections of bundles $\ba\times_P \fg/\fp$ and $\ba\times_P \fp_+$, respectively, and there are corresponding $G_0$-equivariant functions $(\xi)_\si\colon \ba_0\to \fg_-$ and $(\U)_\si\colon \ba_0\to \fp_+$.

Weyl structures always exist on parabolic geometries and for each two Weyl structures~$\si$ and~$\hat \si$, there exist a $1$-form~$\U$ and $G_0$-equivariant functions $\U_i\colon \ba_0 \to \fg_i$ for $i=1,\dots, k$ such that
\begin{gather*}
\hat \si=\si\exp(\U)_\si=\si\exp(\U_1) \cdots \exp(\U_k).
\end{gather*}
The $G_0$-equivariant function $(\U)_\si\colon \ba_0 \to \fp_+$ is related to the functions $\U_i$ via the Baker--Campbell--Hausdorf\/f (BCH)-formula.

We can decompose the pullback $\si^*\om\colon T\ba_0\to \fg$ into $G_0$-equivariant $1$-forms \smash{$\om_{i}^\si\colon T\ba_0\to \fg_i$} according to the grading $\fg_i$ of~$\fg$. These forms clearly depend on the choice of the Weyl struc\-tu\-re~$\si$. For a Weyl structure $\hat \si=\si \exp (\U)_\si$, there is the following formula describing the change of the forms
\begin{gather} \label{change-form}
\om_{l}^{\si\exp(\U)_\si}=\sum_{|i|+j=l} \frac{(-1)^{i}}{i!} \big( \ad(\U_k)^{i_k} \circ \cdots \circ \ad(\U_1)^{i_1}\big)\circ \om_{j}^\si,
\end{gather}
where we write $i!=i_1!\cdots i_k!$, $|i|=i_1+2i_2+ \cdots +ki_k$ and $(-1)^i=(-1)^{i_1+\cdots+i_k}$ for the multi-index $i=(i_1, \dots, i_k)$ with $i_1, \dots, i_k \geq 0$.

The sum $\om_-^\si$ of the forms $\om_{i}^\si$ for $i<0$ is called the \emph{soldering form} given by the Weyl structu\-re~$\si$. Suppose $(\xi)_{\si}=\xi_{-k}+ \cdots+ \xi_{-1}$ holds for the vector f\/ield $\xi$ on $M$ and for $G_0$-equivariant functions $\xi_i\colon \ba_0\to \fg_i$. If $(\xi)_{\si \exp (\U)_\si}=\hat \xi_{-k}+\cdots+\hat \xi_{-1}$ holds for $G_0$-equivariant functions $\hat \xi_i: \ba_0\to \fg_i$ and the Weyl structure $\si\exp (\U)_\si$, then
\begin{gather*}
\hat \xi_l=\sum_{|i|+j=l} \frac{(-1)^i}{i!} \ad(\U_k)^{i_k} \circ \cdots \circ \ad(\U_1)^{i_1}.\xi_{j},
\end{gather*}
where $.$ is the algebraic action of the values of functions $\ba_0\to \fp_+$ on the values of the functions $\ba_0\to \fg_-$.

The form $\om_{0}^\si$ is a principal connection form on $\ba_0$. Suppose that the f\/inite-dimensional representation of $P$ on $V$ is completely reducible as a representation of $G_0$. Then
\begin{enumerate}\itemsep=0pt
\item[1)] the form $\om_{0}^\si$ induces a linear connection $\nabla^\si$ on the space of $P$-equivariant functions $\ba\to V$,
\item[2)] for each $P$-equivariant function $\tau\colon \ba\to V$, the connection $\nabla^\si$ preserves the decomposition of $(\tau)_\si$ into $G_0$-equivariant components.
\end{enumerate}

The induced connections $\na^\si\!$ on $\mathcal{V}$ are called \emph{Weyl connections}. The Weyl connection $\na^{\si\exp (\U)_\si}\!\!$ on~$\mathcal{V}$ is related to the Weyl connection $\na^\si$ on $\mathcal{V}$ by
\begin{gather} \label{change-wc}
\big(\na^{\si\exp (\U)_\si }_\xi\tau\big)_\si =(\na^\si_\xi\tau)_\si + \sum_{|i|+j=0} \frac{(-1)^{i}}{i!} \big( \ad(\U_k)^{i_k} \circ \cdots \circ \ad(\U_1)^{i_1}(\xi_j)\big). (\tau)_\si,
\end{gather}
where $\tau$ is a section of $\mathcal{V}$ and $.$ is the algebraic action of the values of functions $\ba_0\to \fg_0$ on the values of the function $(\tau)_\si\colon \ba_0\to V$.

The soldering form $\om_-^\si$ together with the principal connection form $\om_{0}^\si$ form the Cartan connection $\om_-^\si\oplus\om_{0}^\si$ on $\ba_0$ of a reductive type. In fact, we can view the f\/irst-order frame bund\-le~$\mathcal{P}^1M$ as the bundle $\ba\times_{\un \Ad}{\rm Gl}(\fg/\fp)$ for the adjoint action $\un \Ad$ of $P$ on $\fg/\fp$. Moreover, each Weyl structure $\si$ provides a reduction $\iota_\si\colon \ba_0\to\mathcal{P}^1M$ over $\un \Ad\colon G_0\to {\rm Gl}(\fg/\fp)$ such that
\begin{gather*}\iota_\si^*\theta=\om_-^\si \qquad \text{and} \qquad \iota_\si^*\gamma_\si=\om_{0}^\si\end{gather*}
hold for the natural soldering form $\theta$ on $\mathcal{P}^1M$ and the principal connection form $\gamma_\si$ on $\mathcal{P}^1M$ of the Weyl connection $\nabla^\si$. This allows us to describe explicitly geodesics of Weyl connections. The geodesic of the Weyl connection $\nabla^\si$ on $TM$ through $x$ in the direction $\xi(x) \in T_xM$ is the curve \begin{gather}\label{geod}
p_0 \circ \Fl^{(\om_-^\si\oplus\om_{0}^\si)^{-1}(\xi(x))_\si(u_0)}_t(u_0)
\end{gather} for arbitrary $u_0\in \ba_0$ in the f\/iber over $x$. Indeed, since $(\om_-^\si\oplus\om_{0}^\si)^{-1}((\xi(x))_\si)$ is contained in the kernel of the connection form $\om_0^\si=\iota_\si^*\gamma_\si$ and $Tp_0\circ (\om_-^\si\oplus\om_{0}^\si)^{-1}((\xi(x))_\si)(x)=\xi(x)$, the claimed curve is the projection of a f\/low of a standard horizontal vector f\/ield of $\gamma_\si$ and therefore a geodesic of $\nabla^\si$.

\subsection{The characterization of automorphisms}\label{sec2.2}
Let $\va\colon \ba \to \ba$ be a (local) automorphism of the parabolic geometry and denote by $\va_0:\ba_0 \to \ba_0$ the underlying (local) $G_0$-bundle morphism. Then for each Weyl structure~$\si$, there is a $1$-form~$\U^{\si,\va}$ on~$M$ such that
\begin{gather} \label{action}
\va(\si(u_0))=\si(\va_0(u_0))\exp((\U^{\si,\va})_\si(u_0))
\end{gather}
holds for all $u_0\in \ba_0$. Consequently, the pullback of a Weyl structure is again a Weyl structure, i.e., \begin{gather*}\va^*\si=\va^{-1} \circ \si \circ \va_0=\si \exp(-(\U^{\si,\va})_\si).\end{gather*}

\begin{Lemma} \label{lem3}
Let $\va\colon \ba \to \ba$ be a $($local$)$ automorphism. Then
\begin{gather} \label{2.4.}
\big(\U^{\si\exp (\U)_\si,\va}\big)_{\si\exp (\U)_\si}=C(-(\U)_\si\circ\va_0,C((\U^{\si,\va})_\si,(\U)_\si))
\end{gather}
holds for the Weyl structure $\si\exp (\U)_\si$, where $C$ represents the BCH-formula.
\end{Lemma}
\begin{proof}
We get immediately from the formula (\ref{action}) that
 \begin{gather*} \va(\si(u_0))\exp((\U)_\si(u_0))=\si(\va_0(u_0))\exp((\U)_\si(\va_0(u_0)))\exp\big(\big(\U^{\si\exp (\U)_\si,\va}\big)_{\si\exp (\U)_\si}(u_0)\big)
\end{gather*} holds for all $u_0\in \ba$. This implies \begin{gather*}\exp\big(\big(\U^{\si\exp (\U)_\si,\va}\big)_{\si\exp (\U)_\si}\big)=\exp(-(\U)_\si\circ \va_0)\exp((\U^{\si,\va})_\si)\exp((\U)_\si),\end{gather*} which gives the formula.
\end{proof}

Therefore if $f=\un{\va}$ for a (local) automorphism $\va$ of the parabolic geometry, then
\begin{gather*}f^*\nabla^\si=\nabla^{\si \exp(-(\U^{\si,\va})_\si)}\end{gather*}
holds for each Weyl connection $\nabla^\si$.

There is a unique lift $\mathcal{P}^1f$ of each (local) dif\/feomorphism $f$ on $M$ to the (local) ${\rm Gl}(\fg/\fp)$-bundle morphism on $\mathcal{P}^1M$ such that $(\mathcal{P}^1f)^*\theta=\theta$ holds. If $f^*\nabla^{\si}=\nabla^{\si'}$ is satisf\/ied for some Weyl connections $\nabla^{\si}$ and $\nabla^{\si'}$, then $(\mathcal{P}^1f)^*\gamma_\si=\gamma_{\si'}$ holds. However, this does not imply that such $f$ preserves the parabolic geometry. The (local) dif\/feomorphisms $f$ that preserve the parabolic geometry also satisfy that
\begin{gather*}\mathcal{P}^1f(\iota_{\si'}(\ba_0))=\iota_\si(\ba_0)\end{gather*} holds for reductions $\iota_\si(\ba_0)$ and $\iota_{\si'}(\ba_0)$ of $\mathcal{P}^1M$ and it turns out that this is the crucial property that distinguishes the dif\/feo\-mor\-phisms preserving the parabolic geometry among all dif\/feo\-mor\-phisms preserving the set of all Weyl connections.

\begin{Proposition} \label{auto}
Let $f$ be a $($local$)$ diffeomorphism on $M$ such that for some Weyl structu\-res~$\si$ and~$\si'$ of the parabolic geometry $(\ba\to M,\om)$
\begin{itemize}\itemsep=0pt
\item $f^*\nabla^{\si}=\nabla^{\si'}$ holds, and
\item $\mathcal{P}^1f$ maps a point of $\iota_{\si'}(\ba_0)$ into the image $\iota_{\si}(\ba_0)$.
\end{itemize}
Then $f$ preserves the parabolic geometry.
\end{Proposition}
\begin{proof}
The assumptions imply that $\va_0:= \iota_{\si}^{-1}\circ \mathcal{P}^1f\circ \iota_{\si'}$ is a well-def\/ined (local) $G_0$-bundle morphism $\va_0\colon \ba_0\to \ba_0$ satisfying $\va_0^*\om_{0}^\si=\om_{0}^{\si'}$ and $\va_0^*\om_-^\si=\om_-^{\si'}$. The associated graded map $(\theta_{-k},\dots,\theta_{-1})\colon T\ba_0\to \fg_{-k}\oplus \dots \oplus \fg_{-1}$ corresponding to $\om_-^\si$ is independent of the choice of the Weyl structure according to the formula~(\ref{change-form}). In fact, the tuple $(p_0\colon \ba_0\to M, (\theta_{-k},\dots,\theta_{-1}))$ is a regular inf\/initesimal f\/lag structure with a (local) automorphism~$\va_0$, see \cite[Section~3.1.6--3.1.8]{parabook}. Therefore the claim of theorem follows from \cite[Theorem~3.1.14]{parabook} except for projective and contact projective geometries. In the case of projective geometries, the claim trivially follows from the assumption $f^*\nabla^{\si}=\nabla^{\si'}$. In the case of contact projective geometries, $\va_0$ is a (local) automorphism of the regular inf\/initesimal f\/lag structure if and only if $f$ is a contactomorphism and the claim again follows from $f^*\nabla^{\si}=\nabla^{\si'}$, see \cite[Section 4.2]{parabook} for details.
\end{proof}

\subsection{Normal Weyl structures and generalized geodesics} There is a distinguished class of local Weyl structures, so-called \emph{normal Weyl structures at $x=p(u)$}, each of which is determined by a choice of $u\in \ba$.
More precisely, we consider local Weyl structures $\nu_u$ given by
\begin{gather*}\nu_u\big(\pi\big(\Fl_1^{\om^{-1}(X)}(u)\big)\big):=\Fl_1^{\om^{-1}(X)}(u)\end{gather*} for $X$ in some neighbourhood of $0$ in $\fg_-$. The Weyl structures $\nu_u$ for all $u\in \ba_x$ exhaust all normal Weyl structures at $x$, see \cite[Section~5.1.12]{parabook}.
These Weyl structures are distinguished by the fact that
\begin{gather}\label{autpos}
\va\big(\Fl^{\om^{-1}(X)}_1(u)\big)=\Fl^{\om^{-1}(X)}_1(\va(u))
\end{gather}
holds for all (local) automorphisms $\va$ of the parabolic geometry and all $X$ in some neighbourhood of $0$ in $\fg$. This particularly means that
\begin{gather*}\va^*\nu_u=\nu_{\va^{-1}(u)}\end{gather*}
holds for all (local) automorphisms $\va$ of parabolic geometries.

The curves of the form
\begin{gather*}p\circ \Fl_t^{\om^{-1}(X)}(u)\end{gather*}
for $X\in \fg_-$ and $u \in \ba$ are called \emph{generalized geodesics.} They always provide the normal coordinate system given by~$u$. The crucial observation is that the set of generalized geodesics going through~$x$ coincides with the set of geodesics of normal Weyl connections $\nabla^{\nu_u}$ for all~$u$. Therefore there is the following description of automorphisms of parabolic geometries.

\begin{Proposition}\label{prop2.4}
Let $\va$ be a $($local$)$ $P$-bundle morphism on $\ba$ and let $f=\un{\va}$ be its underlying $($local$)$ diffeomorphism of~$M$. If $\va$ is a~$($local$)$ automorphism of the parabolic geometry, then the equality $f^*\nabla^{\nu_u}=\nabla^{\nu_{\va^{-1}(u)}}$ holds for all $u\in \ba$ and $f$ maps the set of generalized geodesics going through $x$ onto the set of generalized geodesics going through~$f(x)$.

Moreover, if $f$ has coordinates $\Ad(g_0)\in {\rm Gl}(\fg_-)$ for $g_0\in G_0$ in the normal coordinate system given by $u\in \ba$, then $\va$ is a $($local$)$ automorphism of the parabolic geometry if and only if $f^*\nabla^{\nu_u}=\nabla^{\nu_{u}}$ holds.
\end{Proposition}
\begin{proof}
Since $f^*\nabla^\si=\nabla^{\va^*\si}$ holds for all Weyl structures $\si$ and all (local) automorphisms $\va$ of the parabolic geometry, the f\/irst claim follows from the formula~(\ref{geod}). If $f$ has coordinates $\Ad(g_0)\in {\rm Gl}(\fg_-)$ in the normal coordinate system given by $u\in \ba$, then the second assumption of Proposition \ref{auto} is satisf\/ied. Then the second claim is a consequence of the f\/irst claim and Proposition~\ref{auto}, because $\nabla^{\nu_u}=\nabla^{\nu_{ug_0^{-1}}}$ holds.
\end{proof}

\section[The uniqueness of $s$-symmetries and the prolongation rigidity]{The uniqueness of $\boldsymbol{s}$-symmetries and the prolongation rigidity}\label{SEC3}

In this section, we prove Theorem \ref{t1.2}. We also characterize all triples $(\fg,\fp,\mu)$ that are prolongation rigid outside of the $1$-eigenspace of $s$.

\subsection[Consequences of the existence of more $s$-symmetries at one point]{Consequences of the existence of more $\boldsymbol{s}$-symmetries at one point}\label{sec3.1}

Let us recall that if $V$ is an irreducible $G_0$-module, then the element $s\in Z(G_0)$ acts on $V$ by a~single eigenvalue. In particular, we can decompose each completely reducible $G_0$-module $V$ into $G_0$-submodules \begin{gather*}V^s(a):=\{X\in V\colon s(X)=aX\}\end{gather*}
according to the eigenvalues of the action of $s\in Z(G_0)$. In particular, we will often consider the $1$-eigenspaces $\fg_-^s(1)$, $\fg_i^s(1)$ and $\fp_+^s(1)$ in $\fg_-$, $\fg_i$ and $\fp_+$, respectively.

The following proposition is a crucial technical result for the proof of Theorem~\ref{t1.2}.

\begin{Proposition} \label{auto-action}
Let $s_u$ be a $($local$)$ $s$-symmetry at $x$ for some $u \in \ba_x$. Then for each Weyl structure $\si$, there is a $1$-form $\U^{\si,s_u}$ on $M$ satisfying
\begin{enumerate}\itemsep=0pt
\item[$1)$] $s_u^*\si=\si\exp(-(\U^{\si,s_u})_\si)$,
\item[$2)$] $(\U^{\si,s_u})_\si(\pi(u))=C(-\Ad(s)^{-1}(Y),Y)$ for some $Y\in \fp_+$,
where $C$ represents the BCH-formula on the nilpotent Lie algebra $\fp_+$, and
\item[$3)$] if $(\U^{\si,s_u})_\si(\pi(u))=Z_{i}+\cdots+Z_k$ holds for $Z_{j}\in \fg_j$, then the component of $Z_{i}$ contained in $\fg_i^s(1)$ is trivial, where $i$ is the smallest index $j$ such that $(\U^{\si,s_u})_\si(\pi(u))$ has a non-zero component in $\fg_i$.
\end{enumerate}
Moreover, if $s_v$ is a $($local$)$ $s$-symmetry at $x$ for some $v\in \ba_x$, then $s_u=s_v$ if and only if $\U^{\si,s_u}(x)=\U^{\si,s_v}(x)$ holds.
\end{Proposition}
\begin{proof}
The normal Weyl structure $\nu_u$ always satisf\/ies $\nu_u(\pi(u))=u$ and therefore the set of all Weyl structures $\si$ satisfying $\si(\pi(u))=u$ is non-empty. Let $s_u$ be a (local) $s$-symmetry at~$x$ and consider arbitrary Weyl structure $\si$ satisfying $\si(\pi(u))=u$. Then $(\U^{\si,s_u})_\si(\pi(u))=0$ holds and the Lemma~\ref{lem3} implies that $(\U^{\si,s_u})_\si$ has the claimed properties~(1) and~(2) for arbitrary Weyl structure. The claimed property (3) holds, because the BCH-formula implies that $C(-\Ad(s)^{-1}(Y),Y)_i=-\Ad(s)^{-1}(Y_i)+Y_i=Z_i$ holds.

If $s_v$ is a (local) $s$-symmetry at $x$ for some $v\in \ba_x$, then $s_u=s_v$ if and only if $us=s_u(u)=s_v(u)$ holds. Thus we need to show that if $\U^{\si,s_u}(x)=\U^{\si,s_v}(x)$ holds, then \smash{$s_u=s_v$}. We can assume $\si(\pi(u))=u$ for the Weyl structure $\si$, because the equality $\U^{\si,s_u}(x)=\U^{\si,s_v}(x)$ is preserved if we change the Weyl structure $\si$. Suppose $g_0\in G_0$ and $Y\in \fp_+$ are such that $v=ug_0\exp(Y)$ holds. If $\hat \si$ is a Weyl structure such that $\hat \si(\pi(u))=ug_0\exp(Y)$, then $\U^{\si,s_u}(x)= 0$, $\U^{\hat \si,s_v}(x)=0$ and $(\U^{\si,s_v})_\si(\pi(u))=C(-\Ad(s)^{-1}(\Ad(g_0)(Y)),\Ad(g_0)(Y))$ hold. Since $C(-\Ad(s)^{-1}(\Ad(g_0)(Y)),\Ad(g_0)(Y))=0$ if and only if $\Ad(s)(Y)=Y$, the element~$s$ commutes with $g_0\exp(Y)$ and $s_u=s_v$ holds.
\end{proof}

The harmonic curvature $\ka_H$ is preserved by each (local) automorphism of the parabolic geometry. Since~$\kappa_H$ is a section of an associated vector bundle to~$\ba$ for a representation of~$P$ which is trivial on $\exp(\fp_+)$, the function $(\kappa_H)_\si$ does not depend on the choice of the Weyl structure~$\si$ and we will write $\kappa_H(u)$ instead of $(\kappa_H)_\si(\pi(u))$. Consequently, $\kappa_H(p(u))=0$ if and only if $\kappa_H(u)=0$.

If $s_u$ is a (local) $s$-symmetry at $p(u)$, then $\un s_u^*\kappa_H=\kappa_H$. Thus $s.\kappa_H(u)=\kappa_H(u)$ trivially follows, where we denote by $.$ the tensorial action of $\fg_0$ on $\kappa_H$. This proves the f\/irst claim of Theorem \ref{t1.2}.

The second claim of Theorem \ref{t1.2} is a consequence of the following proposition and Def\/i\-ni\-tion~\ref{prolong} of the prolongation rigidity.

\begin{Proposition} \label{cor-ws}
Assume there are $($local$)$ $s$-symmetries $s_u$ and $s_v$ at $x$ for some $u,v\in \ba_x$. Suppose that $(\U^{\si,s_v})_\si(\pi(u))=0$
and $(\U^{\si,s_u})_\si(\pi(u))=Z_i+\cdots+Z_k$ hold for some Weyl structu\-re~$\si$. Then $Z_i\in \operatorname{pr}(\kappa_H(u))_i$.
\end{Proposition}
\begin{proof}
We show that $\ad(X_1)\cdots \ad(X_i)(Z_i).\kappa_H(u)=0$ holds for all $X_1,\dots, X_i\in\fg_{-1}$. Consider an arbitrary Weyl structure~$\si$ and consider the iterated covariant derivative $(\nabla^\si)^j_{\xi^1,\dots,\xi^j}$ for vector f\/ields $\xi^1,\dots,\xi^j$ such that
\begin{gather*}\big(\xi^b\big)_\si=\xi^b_{-1}\colon \ \ba_0\to \fg_{-1},\qquad
\big(\xi^b\big)_\si(\pi(u))=X^b\end{gather*} hold for some $X^b\in \fg_{-1}^s( {1 \over a_b})$ for some $a_b$ for all $1\leq b \leq j$. We assume $j\leq i$ unless we state otherwise.

We compute
\begin{gather*}
(\un s_u^*\nabla^\si)^j_{\xi^1,\dots,\xi^j} \ka_H(u)= \un s_u^*(\nabla^\si)^j_{(\un s_u)_*\xi^1,\dots,(\un s_u)_*\xi^j}(\un s_u)_*\ka_H(u)
=(\nabla^\si)^j_{(\un s_u)_*\xi^1,\dots,(\un s_u)_*\xi^j}\ka_H(u).
\end{gather*}
Since we assume $X^b\in \fg_{-1}^s({1 \over a_b})$, we get
\begin{gather*}
\big((\un s_u)_*\xi^b\big)_{\si \exp(\U^{\si,s_u})_\si}(\pi(u))=\big((\un s_u)_*\xi^b\big)_{\si}(\pi(u))=\big(\xi^b\big)_\si(\pi(u)s)\\
\hphantom{\big((\un s_u)_*\xi^b\big)_{\si \exp(\U^{\si,s_u})_\si}(\pi(u))}{}=\Ad(s)^{-1}\big(\xi^b\big)_\si(\pi(u))=a_bX^b.
\end{gather*}
Thus
\begin{gather}\label{form-act}
(\un s_u^*\nabla^\si)^j_{\xi^1,\dots,\xi^j} \ka_H(u)=a_1\cdots a_j(\nabla^\si)^j_{\xi^1,\dots,\xi^j} \ka_H(u).
\end{gather}

If $(\U^{\si,s_u})_\si(\pi(u))=Z_i+\dots+Z_k$ holds for the Weyl structure $\si$, then the formula~(\ref{change-wc}) together with Proposition~\ref{auto-action} imply
\begin{gather*}(\un s_u^*\na^\si)_{\xi^b}\ka_H(u)=\na^{\si \exp(-(\U^{\si,s_u})_\si)}_{\xi^b}\ka_H(u)=\na^\si_\xi \ka_H(u)+\ad(Z_i)\big(X^{b}\big).\ka_H(u).\end{gather*}
In particular, if $i>1$, then
\begin{gather*}(\un s_u^*\na^\si)_{\xi^b}\ka_H(u)=\na^{\si \exp(-(\U^{\si,s_u})_\si)}_{\xi^b}\ka_H(u)=\na^\si_\xi \ka_H(u).\end{gather*}
If we apply the above formulas onto the f\/irst connection in $(\un s_u^*\nabla^\si)^j_{\xi^1,\dots,\xi^j} \ka_H(u)$, then we obtain
\begin{gather*}
(\un s_u^*(\nabla^\si)^j)_{\xi^1,\dots,\xi^j} \ka_H(u) =\na^\si_{\xi^1}(\un s_u^*\nabla^\si)^{j-1}_{\xi^2,\dots,\xi^{j}} \ka_H(u).
\end{gather*}
In the next step, the same formulas for the second connection lead to the formula
\begin{gather*}
(\un s_u^*\nabla^\si)^j_{\xi^1,\dots,\xi^j} \ka_H(u) =(\na^\si)^2_{\xi^1,\xi^2}(\un s_u^*\nabla^\si)^{j-2}_{\xi^3,\dots,\xi^{j}} \ka_H(u)\\
\hphantom{(\un s_u^*\nabla^\si)^j_{\xi^1,\dots,\xi^j} \ka_H(u) =}{}
-\ad(X^{2})((\nabla^\si)_{\xi^1}(\U^{\si,s_u})_\si).(\un s_u^*\nabla^\si)^{j-2}_{\xi^1,\dots,\xi^{l-j}}\ka_H(u).
\end{gather*}
Thus before we consider the next step, we need to characterize the components of $(\na^\si_{\xi^b}\! \U^{\si,s_u}\!)_{\si(\pi(u))\!}\!$ in $\fg_1\oplus \cdots \oplus \fg_{j}$ for $j<i$. Firstly, let us view $(\U^{\si,s_u})_\si$ as a section of the adjoint tractor bundle $\ba\times_P \fg$. Observe that the covariant derivative $\na^\si_{\xi^b}$ coincides with the fundamental derivative on the components in $\fg_-\oplus \fg_0 \oplus \fg_1\oplus \dots \oplus \fg_j$ according to the formula from \cite[Proposition~5.1.10]{parabook}. We know that $(\U^{\si,s_u})_\si$ has its values in $\fp_+$ and the components of $(\na^\si_{\xi^b} \U^{\si,s_u})_{\si(\pi(u))}$ in $\fg_1\oplus \cdots \oplus \fg_{j}$ for $j<i$ are tensorial both in~$\xi^b$ and $\U^{\si,s_u}$. Then, using the formula from \cite[Corollary~1.5.8]{parabook} and the $P$-equivariancy of $\om$, we get the following equality on the restriction to $\fg_1\oplus \cdots \oplus \fg_j$ for $j<i$
\begin{gather*}
\big(\na^\si_{\xi^b} \U^{\si,s_u}\big)_{\si(\pi(u))} =\om(\si(\pi(u)))\big(\big[\om^{-1}\big(\xi^b\big),\om^{-1}(Z_i)\big]\big)=-\ad\big(X^b\big)(Z_i).
\end{gather*}
Therefore
\begin{align*}
(\un s_u^*\nabla^\si)^j_{\xi^1,\dots,\xi^j} \ka_H(u)&=(\na^\si)^2_{\xi^1,\xi^2}(\un s_u^*\nabla^\si)^{j-2}_{\xi^3,\dots,\xi^{j}} \ka_H(u).
\end{align*}

If we iterate the computation of $(\na^\si_{\xi^b} \U^{\si,s_u})_{\si(\pi(u))}$ for $j<i$, then we obtain by the same arguments
\begin{gather*}\big((\nabla^\si)^j_{\xi^1,\dots,\xi^j} \U^{\si,s_u}\big)_{\si(\pi(u))}=(-1)^j\ad\big(X^j\big)\cdots \ad\big(X^1\big)(Z_i)\end{gather*}
for the component in $\fg_1\oplus \cdots \oplus \fg_{i-j}$. Thus for $j<i$, we obtain
\begin{gather*}
(\un s_u^*\nabla^\si)^j_{\xi^1,\dots,\xi^j} \ka_H(u) =(\nabla^\si)^{j}_{\xi^1,\dots,\xi^{j}} \ka_H(u)-\ad(X^{j})((\nabla^\si)^{j-1}_{\xi^1,\dots,\xi^{j-1}}(\U^{\si,s_u})_\si).\ka_H(u)\\
\hphantom{(\un s_u^*\nabla^\si)^j_{\xi^1,\dots,\xi^j} \ka_H(u)}{}
=(\nabla^\si)^{j}_{\xi^1,\dots,\xi^{j}} \ka_H(u)
\end{gather*}
and for $j=i$, we obtain
\begin{gather} \label{form2}
(\un s_u^*\nabla^\si)^i_{\xi^1,\dots,\xi^i} \ka_H(u)=(\nabla^\si)^i_{\xi^1,\dots,\xi^i} \ka_H(u)+(-1)^i\ad\big(X^i\big)\cdots \ad\big(X^1\big)(Z_i).\kappa_H(u).
\end{gather}

If we compare the formulas (\ref{form-act}) and (\ref{form2}) for $(\un s_u^*\nabla^\si)^j_{\xi^1,\dots,\xi^j} \ka_H(u)$, we obtain
\begin{gather}\label{gl--zavorka}
(-1)^j(a_1\cdots a_j-1)(\nabla^\si)^j_{\xi^1,\dots,\xi^j} \ka_H(u) =\ad\big(X^j\big)\cdots \ad\big(X^1\big)(Z_i).\kappa_H(u)
\end{gather}
for all $j\leq i.$

If the Weyl structure $\si$ satisf\/ies $(\U^{\si,s_v})_\si(\pi(u))=0$, then we simultaneously have
\begin{gather*}(a_1\cdots a_j-1)(\nabla^\si)^j_{\xi^1,\dots,\xi^j} \ka_H(u)=0\end{gather*}
for all $j\leq i$ if we follow the proof for $s_v$ instead of $s_u$. Thus if $a_1\cdots a_j-1\neq 0$, then $\ad(X^i)\cdots \ad(X^1)(Z_i).\kappa_H(u)=0$. But since $Z_i$ has a trivial component in~$\fg_i^s(1)$, we know that $\ad(X^i)\cdots \ad(X^1)(Z_i)\neq 0$ implies $a_1\cdots a_j-1\neq 0$ and the claim of the proposition holds due to the linearity.
\end{proof}

If we follow the computations from the proof of Proposition \ref{cor-ws} for a Weyl structure $\si$ satisfying $\U^{\si,s_u}(x)=0$, then most of the assumptions on the vector f\/ields $\xi^b$ are vacuous and $(\un s_u^*\nabla^\si)_{\xi} \ka_H(u)=\nabla^{\si}_{\xi} \ka_H(u)$ holds for arbitrary vector f\/ield $\xi$. Therefore we obtain the following corollary using the formula (\ref{form-act}) for $\xi$ from particular eigenspaces of $T_x\un s_u$.

\begin{Corollary} \label{invariantni-ws}
Let $s_u$ be a $($local$)$ $s$-symmetry at $x=p(u)$ on a parabolic geometry and assume $\U^{\si,s_u}(x)=0$. Then we get
\begin{gather*}
\na^\si_{\xi} \ka_H(x)=\na^\si_{\xi_{\rm f\/ix}}\ka_H(x),
\end{gather*}
where $\xi_{\rm f\/ix}\in T_xM$ is the component of $\xi\in T_xM$ such that $(\xi_{\rm f\/ix})_\si(\pi(u))\in \fg_-^s(1)$. In particular, if $\fg_-^s(1)=0$, then $\na^\si_\xi\ka_H(x)=0$ holds for all $\xi\in T_xM$.
\end{Corollary}

\begin{Remark}The authors showed in \cite{RT} and \cite{DR} that there are projective and conformal geometries satisfying $\na^\si\ka_H(x)=0$ for all $x\in M$ for a suitable Weyl connection $\na^\si$, but $(M,\na^\si)$ are not an af\/f\/ine locally symmetric spaces. Therefore Theorem \ref{t1.1} implies that the condition $\na^\si\ka_H=0$ is necessarily satisf\/ied on (locally) $m$-symmetric parabolic geometries, but is not suf\/f\/icient to distinguish the (locally) $m$-symmetric parabolic geometries among the geometries satisfying $\na^\si\ka_H=0$.
\end{Remark}

\subsection[The characterization of triples that are prolongation rigid outside of the 1-eigenspace of $s$]{The characterization of triples that are prolongation rigid\\ outside of the 1-eigenspace of $\boldsymbol{s}$}\label{sec4.2}

We can estimate the dimension of $\operatorname{pr}(\kappa_H(u))_i$ in the following way: The result of \cite[Proposition~3.1.1]{KT} states that the dimension of $\operatorname{ann}(\kappa_H(u))$ is bounded by the dimension of the annihilator $\mathfrak{a}_0:=\cap_{\phi_0} \operatorname{ann}(\phi_0)$ of all minus lowest weights $\phi_0$ in (the complexif\/ication of) all irreducible $\fg_0$-modules in which $\kappa_H(u)$ has a non-zero component. Moreover, the dimension of $\operatorname{pr}(\kappa_H(u))_i$ is bounded by the dimension of the prolongation $\mathfrak{a}_i:=\cap_{\phi_0} \operatorname{pr}(\phi_0)_i$ of $\mathfrak{a}_0$. The main result of \cite[Theorem~3.3.3 and Recipe~7]{KT} states that there is a semisimple Lie subalgebra $\bar \fg$ of $\fg$ and a~parabolic subalgebra $\bar \fp$ of $\bar \fg$ such that $\mathfrak{a}_i=\bar \fg_i$ for $i>0$.

Let us prove that these estimates are compatible with the decomposition of $\fg_i$ into $\fg_0$-submodules, which allows us to characterize the triples $(\fg,\fp,\mu)$ that are prolongation rigid outside of the $1$-eigenspace of~$s$.

\begin{Proposition} \label{multigrad}
Suppose $Z\in \operatorname{pr}(\kappa_H(u))_i$ decomposes as $Z=Z_a+Z_b$ for $Z_a,Z_b$ in different $\fg_0$-submodules of $\fg_i$. Then $Z_a\in \operatorname{pr}(\kappa_H(u))_i$ and $Z_b\in \operatorname{pr}(\kappa_H(u))_i$.

Therefore the triple $(\fg,\fp,\mu)$ is prolongation rigid outside of the $1$-eigenspace of $s$ if and only if $\mathfrak{a}_i$ corresponding to $\mu$ is a subspace of $\fg_i^s(1)$ for all $i$.
\end{Proposition}
\begin{proof}
Let $(\alpha_1,\dots,\alpha_j)$ be an ordering of simple positive roots of $\fg$ such that the root space~$\fg_{\alpha_r}$ satisf\/ies $\fg_{\alpha_r}\in \fg_1$. Then we can uniquely assign a $j$-tuple $(a_1,\dots,a_j)$ to each irreducible $\fg_0$-component of~$\fg_i$, where $a_\ell$ is the height of all root spaces in the $\fg_0$-component with respect to~$\alpha_\ell$. This def\/ines a multigrading of $\fg$ and the Lie bracket in $\fg$ is multigraded.

Let us decompose the element $Z\in \operatorname{pr}(\kappa_H(u))_i$ as the sum of the elements $\sum Z_{(b_1,\dots,b_j)}$ over all possible $j$-tuples with respect to this multigrading. Similarly, let us decompose the module $\otimes^i \fg_{-1}$ as the sum of modules $\oplus \mathfrak{n}_{(a_1,\dots,a_j)}$ over all possible $j$-tuples with respect to this multigrading.
The multigrading of $\fg_0$ is of the form $(0,\dots, 0)$, and therefore,
\begin{gather*}\ad^i(X)(Z)=\sum \ad^i(X_{(a_1,\dots,a_j)}) \Big(\sum Z_{(b_1,\dots,b_j)}\Big)=\sum \ad^i(X_{(-b_1,\dots,-b_j)}) (Z_{(b_1,\dots,b_j)})\end{gather*}
holds for all $X=\sum X_{(a_1,\dots,a_j)} \in \oplus \mathfrak{n}_{(a_1,\dots,a_j)}$. Thus we get that \begin{gather*}\ad^i\big(X_{(-b_1,\dots,-b_j)}\big) \big(Z_{(b_1,\dots,b_j)}\big)\in \operatorname{ann}(\kappa_H(u))\end{gather*} holds for all $X=X_{(-b_1,\dots,-b_j)}\in \mathfrak{n}_{(-b_1,\dots,-b_j)}$.
Thus $Z_{(b_1,\dots,b_j)}\in \operatorname{pr}(\kappa_H(u))_i$ follows from the linearity for all components $Z_{(b_1,\dots,b_j)}$ of $Z$.

The f\/irst claim implies that the proof of \cite[Proposition~3.1.1]{KT} can be carried separately for each component of $\operatorname{pr}(\kappa_H(u))_i$ in $\fg_0$-submodule in~$\fg_i$ and thus the second claim follows from \cite[Theorem~3.3.3]{KT}.
\end{proof}

One can f\/ind in \cite[Appendix C]{GZ3} tables containing the classif\/ication of the triples $(\fg,\fp,\mu)$ such that $\mu$ is contained in the $1$-eigenspace of $s$ for some $s\in Z(G_0)$ (dif\/ferent from identity), the classif\/ication of the modules $\mathfrak{a}_i$ and the classif\/ication of the $1$-eigenspaces of $s$ in~$\fp_+$. This allows us to classify all triples $(\fg,\fp,\mu)$ that are prolongation rigid outside of the $1$-eigenspace of~$s$ such that $\mu$ is contained in the $1$-eigenspace of~$s$.

We would like to present the classif\/ication together with additional properties of the corresponding (locally) $s$-symmetric parabolic geometries. Therefore we postpone the classif\/ication to Section \ref{sec61} and continue by looking on geometric properties of generic (locally) $s$-symmetric parabolic geometries.

\section{Geometric properties of parabolic geometries of general types} \label{sec3}
We present here geometric properties that are common for (locally) $s$-symmetric parabolic geometries for triples $(\fg,\fp,\mu)$ that are prolongation rigid outside of the $1$-eigenspace of~$s$. In particular, we prove Theorem~\ref{t1.3}. In order to prove that Claim~(3) implies Claim~(2), we discuss in Section~\ref{sec4.1} when a geodesic transformation~$s_x^\si$ of a Weyl connection~$\nabla^\si$ preserves the parabolic geometry. Claim~(1) follows trivially from Claim~(2) and we discuss the remaining implication in Section~\ref{sec4.3}.

\subsection{Automorphisms and normal coordinate systems of Weyl connections}\label{sec4.1}

Let us describe the (local) dif\/feomorphisms $s_x^\si$ in detail. We know from the formula~(\ref{geod}) that the (local) dif\/feomor\-phism~$s^{\si}_x$ of $M$ def\/ined by the formula
\begin{gather}
s^{\si}_x\big(p_0 \circ \Fl^{(\om_-^\si\oplus\om_{0}^\si)^{-1}(\xi(x))_\si(u_0)}_1(u_0)\big): = p_0 \circ \Fl^{(\om_-^\si\oplus\om_{0}^\si)^{-1}\Ad(s)(\xi(x))_\si(u_0)}_1(u_0)\nonumber\\
\hphantom{s^{\si}_x\big(p_0 \circ \Fl^{(\om_-^\si\oplus\om_{0}^\si)^{-1}(\xi(x))_\si(u_0)}_1(u_0)\big)}{}
= p_0 \circ \Fl^{(\om_-^\si\oplus\om_{0}^\si)^{-1}(\xi(x))_\si(u_0)}_1(u_0s)\label{geodf}
\end{gather}
for some $u_0\in (\ba_0)_x$ does not depend on the choice of $u_0\in (\ba_0)_x$. So $s_x^\si$ is the unique (local) dif\/feomorphism with coordinates $\Ad(s)\in {\rm Gl}(\fg_-)$ for $s\in Z(G_0)$ in the normal coordinate system for the Weyl connection $\nabla^\si$ given by some~$u_0\in \ba_0$.

We also know from Proposition \ref{prop2.4} that for a normal Weyl structure $\nu_u$ for $u\in\ba_x$, the equality
\begin{gather*}s^{\nu_u}_x=\un s_u\end{gather*}
holds. Thus $s^{\nu_u}_x$ preserves the parabolic geometry (and therefore $s_u$ is a (local) automorphism of the parabolic geometry) if and only if $(s^{\nu_u}_x)^*\nabla^{\nu_u}=\nabla^{\nu_{u}}$ holds.

The situation is dif\/ferent for a general Weyl structure $\si$ and the following proposition gives a suf\/f\/icient condition for $s^\si_x$ to be a~(local) $s$-symmetry.

\begin{Proposition} \label{geodesic-s-symm}
Assume the $($local$)$ diffeomorphism $s^{\si}_x$ satisfies
\begin{itemize}\itemsep=0pt
\item $(s^{\si}_x)^*\nabla^\si=\nabla^{\si\exp (\U)_\si}$ for some $1$-form $\U$ on $M$,
and
\item $\U(x)=0$.
\end{itemize}
Then $s_{\si(u_0)}$ is a $($local$)$ $s$-symmetry at $x$ for all~$u_0$ in the fiber over~$x$ such that $\U^{\si,s_{\si(u_0)}}=-\U$, and $\un s_{\si(u_0)}=s^{\si}_x$, i.e., $s^{\si}_x$
preserves the parabolic geometry.
\end{Proposition}
\begin{proof}
Suppose $(s^{\si}_x)^*\nabla^\si=\nabla^{\si\exp(\U)_\si}$ holds for $\U$ such that $\U(x)=0$. Then the inclusions $\iota_{\si}$ and $\iota_{\si\exp(\U)_\si}$ of $\ba_0$ into $\mathcal{P}^1M$ coincide in the f\/iber over $x$ by the assumption $\U(x)=0$. Thus the formula~(\ref{geodf}) implies that $\mathcal{P}^1s^{\si}_x$ maps the frames $\iota_{\si}(u_0)=\iota_{\si\exp (\U)_\si}(u_0)$ in the f\/iber over $x$ onto frames $\iota_{\si}(u_0s)=\iota_{\si\exp(\U)_\si}(u_0s)$. Therefore the conditions of Proposition~\ref{auto} are satisf\/ied and $s^{\si}_x$ preserves the parabolic geometry. Since $\U(x)=0$, it follows from Proposition~\ref{auto-action} that the covering of $s^{\si}_x$ maps $\si(u_0)$ onto $\si(u_0)s$ and thus coincides with $s_{\si(u_0)}$ due to the for\-mu\-la~(\ref{autpos}).
\end{proof}

In particular, if there is an $\un S$-invariant class of Weyl connections, then all (local) dif\/feo\-mor\-phisms~$\un S(x)$ for all $x\in M$ satisfy the conditions of Proposition~\ref{geodesic-s-symm} and therefore Claim~(3) of Theorem~\ref{t1.3} implies Claim~(2) of Theorem~\ref{t1.3}.

A consequence of Propositions \ref{geodesic-s-symm} and~\ref{auto-action} is that the condition $\U^{\si,s_u}(p(u))=0$ is necessary for the equality $\un s_u=s^\si_{p(u)}$ to hold for $s$-symmetry $s_u$ at~$p(u)$. On the other hand, it is clear that the condition $\U^{\si,s_u}(p(u))=0$ is far from being suf\/f\/icient. There is the following consequence of the fact that the af\/f\/ine maps are determined by the image of a single point in $\iota_\si(\ba_0)\subset \mathcal{P}^1M$.

\begin{Corollary}\label{c3.3}
Let $s_u$ be a $($local$)$ $s$-symmetry at $x$ and assume $\U^{\si,s_u}\equiv 0$ holds for some Weyl structu\-re~$\si$. Then $\un s_u=s_x^{\si}$.
\end{Corollary}

\subsection[The prolongation rigidity for $s$-symmetric parabolic geometries]{The prolongation rigidity for $\boldsymbol{s}$-symmetric parabolic geometries}\label{sec4.3}

Let $(\fg,\fp,\mu)$ be prolongation rigid outside of the $1$-eigenspace of $s$. Let $U\subset M$ be the open subset of $M$ consisting of points $x$ such that $\kappa_H(x)$ has a non-zero component in the $\fg_0$-module given by $\mu$. If the parabolic geometry is (locally) $s$-symmetric, then there is a unique (local) $s$-symmetry $s_u$ at each point of~$U$, i.e., there is the unique system~$S$ of (local) $s$-symmetries on~$U$. This means that if there is an almost $\un S$-invariant Weyl connection on~$U$, then the system~$\un S$ coincides (due to uniqueness) with the system of (local) dif\/feomorphisms $\un s_u$. We call a Weyl structure $\si$ (almost) $S$-invariant (at $x$) if $\nabla^\si$ is (almost) $\un S$-invariant Weyl connection (at~$x$).

The uniqueness of $s$-symmetries on $U$ has the following consequences in the case $U=M$.

\begin{Proposition} \label{almost-ws}
Assume $(\fg,\fp,\mu)$ is prolongation rigid outside of the $1$-eigenspace of~$s$ and $\kappa_H(x)$ has a non-zero component in the $\fg_0$-module given by $\mu$ at all $x\in M$. Let $S$ be the unique system of $($local$)$ $s$-symmetries on the $($locally$)$ $s$-symmetric parabolic geometry $(\ba\to M,\om)$ of type~$(G,P)$. Then:
\begin{enumerate}\itemsep=0pt
\item[$1.$] There exists an almost $S$-invariant Weyl structure $\si$ and the map $S$ is smooth.
\item[$2.$] If $\si$ is an almost $S$-invariant Weyl structure, then $\si\exp(\U)_\si$ is an almost $S$-invariant Weyl structure if and only if $(\U)_\si$ has its values in $\fp^s_+(1)$.
\item[$3.$] For each $x\in M$, there is a local almost $S$-invariant Weyl structure $\si$, which is invariant at $x$, and $\un S(x)=s_x^{\si}$ holds.
\item[$4.$] The equality $S(p_0(u_0))=s_{\si(u_0)}$ holds for each almost $S$-invariant Weyl structure $\si$ for all $u_0\in \ba_0$.
\item[$5.$] The equality $S(x)\circ S(y)\circ S(x)^{-1}=S(\un S(x)(y))$ holds for $x,y\in M$, where the compositions are defined.
\item[$6.$] For each eigenvalue $a$, the union of the $a$-eigenspaces $T_xM^s(a)$ of $T_x\un S(x)$ in $T_xM$ over all $x\in M$ defines a distribution $TM^s(a)$ on $M$ that is preserved by all $($local$)$ $s$-symmetries for each~$a$.
\item[$7.$] The equality $TM^s(a)=Tp_0\circ (\om^\si_-+\om^\si_0)^{-1}(\fg_-^s(a))$ holds for each almost $S$-invariant Weyl structure~$\si$.
\item[$8.$] The decomposition $TM=\oplus_a TM^s(a)$ is preserved by all almost $\un S$-inva\-riant Weyl connections $\nabla^\si$.
\item[$9.$] All almost $\un S$-invariant Weyl connections restrict to the same partial linear connection on~$TM$ corresponding to the distribution $\oplus_{a\neq 1} TM^s(a)$.
\end{enumerate}
\end{Proposition}
We show that Claim (1) of Theorem \ref{t1.3} implies Claim (3) of Theorem \ref{t1.3} and simultaneously obtain all the claims of the proposition.
\begin{proof}
Let us pick an arbitrary Weyl structure $\hat \si$ and consider the $G_0$-equivariant function $(S)_{\hat \si}\colon \ba_0\to \{C(-\Ad(s)^{-1}(Y),Y), Y\in \fp_+\}$ def\/ined by \begin{gather*}S(p_0(u_0))^*\hat \si(u_0)=\hat \si(u_0)\exp(-(S)_{\hat \si}(u_0))\end{gather*}
for all $u_0\in \ba_0$. We show that $(S)_{\hat \si}$ is smooth.

We decompose
\begin{gather*}(S)_{\hat \si}=\sum_a\tau_i(a)+\dots+\sum_a\tau_k(a)\end{gather*}
according to the grading and the eigenvalues $a$ of $\Ad(s)$. It follows from Claim~(3) of Proposi\-tion~\ref{auto-action} that $\tau_i(1)\equiv 0$. Thus the formula~(\ref{gl--zavorka}) from the proof of Proposition~\ref{cor-ws} that holds under our assumptions at each point of $M$ implies that each $\tau_i(a)$ is smooth.

The formula (\ref{2.4.}) from the Lemma \ref{lem3} gives \begin{gather*}(S)_{\hat \si\exp (\U)_{\hat \si}}=C\big({-}\Ad(s)^{-1}(\U)_{\hat \si},C((S)_{\hat \si},(\U)_{\hat \si})\big).\end{gather*}
If we take $\U=r\tau_i(a)$ for arbitrary $r\in \mathbb{R}$, then
\begin{gather*}C\big({-}\Ad(s)^{-1}(r \tau_i(a)),C((S)_{\hat \si},r\tau_i(a))\big)_i(a) = C\left(-\frac{r}{a} \tau_i(a),C(\tau_i(a),r\tau_i(a))\right)_i(a)\\
\hphantom{C\big({-}\Ad(s)^{-1}(r \tau_i(a)),C((S)_{\hat \si},r\tau_i(a))\big)_i(a)}{} =
\frac{r(1-a)+a}{a}\tau_i(a)
\end{gather*}
holds for the component of the BCH-formula in $\fg_i(a)$, while the components of the BCH-for\-mu\-la in $\fg_i(b)$ for the other eigenvalues $b\neq a$ of $\Ad(s)$ remain $\tau_i(b)$. Consequently, if we take \begin{gather*}\U_i:=\sum_{a\neq 1} \frac{a}{a-1}\tau_i(a) \end{gather*} and consider the Weyl structure $\hat \si\exp (\U_i)$ instead of $\hat \si$, then we get
\begin{gather*}(S)_{\hat \si\exp (\U_i)}=\sum_a\tilde \tau_{i+1}(a)+\dots+\sum_a\tilde \tau_k(a).\end{gather*}
By induction, we obtain in f\/initely many steps a Weyl structure $\si$ such that $(S)_\si\equiv 0$ holds. Since $(S)_\si\equiv 0$ and all the changes we made are smooth, the function $(S)_{\hat \si}$ and the Weyl structure $\si=\hat \si\exp(\U_i)\cdots \exp(\U_k)$ are smooth.
Let $[\nabla^\si]$ be the class consisting of all Weyl connections for Weyl structures $\si$ constructed for all Weyl structures $\hat \si$. We complete the proof by showing that $[\nabla^\si]$ is an $\un S$-invariant class of Weyl connections and thus Claim~(1) holds.

It is clear from the construction of $\si$ that if we start with $\hat \si\exp (\U)_\si$ for $(\U)_\si$ with values in~$\fp^s_+(1)$, then we get $\si\exp(\U)_\si.$ Thus the class $[\nabla^\si]$ satisf\/ies Claim~(2) and Claims~(6),~(7),~(8) and~(9) are then consequences of Claim~(2) and the formulas for the change of Weyl structures and connections. In particular, the class $[\nabla^\si]$ is a maximal subclass of the class of Weyl connections that satisfy $T_xs_x^{\si}=T_xs_x^{\si'}$ for all Weyl connections $\nabla^\si,\nabla^{\si'}\in [\nabla^\si]$ and all $x\in M$, and that all connections in $[\nabla^\si]$ restrict to the same partial connection on all smooth subbundles of~$TM$ for all eigenvalues of $T_xs_x^\si$ dif\/ferent from~$1$.

If $\hat \si=\nu_u$ is the normal Weyl structure for $u\in \ba_x$ satisfying $S(x)(u)=us$, then
\begin{gather*}
\si\exp\big({-}\big(\U^{\si,S(x)}\big)_\si\big)= \nu_u\exp(S(x)^*\U_i)\cdots \exp(S(x)^*\U_k)\\
\hphantom{\si\exp\big({-}\big(\U^{\si,S(x)}\big)_\si\big)}{}
= \si\exp(-\U_k)\cdots \exp(C(-\U_i,S(x)^*\U_i))\cdots \exp(S(x)^*\U_k).
\end{gather*}
Since the component of $C(-\U_i,S(x)^*\U_i)$ contained in $\fg_i$ has a trivial component in $\fp^s_+(1)$ and $(\U^{\si,S(x)})_\si$ has its values in $\fp^s_+(1)$, the equality $\U_i=S(x)^*\U_i$ holds. Thus we get $0=C(-\U_i,S(x)^*\U_i)$. Therefore $\si\exp(-(\U^{\si,S(x)})_\si)=\si$ follows by induction, and thus $S(x)^*\si=\si$. Corollary~\ref{c3.3} and the last claim of Proposition~\ref{auto-action} implies that
\begin{gather*}\un S(x)=\un s_u=s_x^{\nu_u\exp(\U_i)\cdots \exp(\U_k)}=\un s_{\si(\pi(u))}\end{gather*}
holds for all $x\in M$, all $u\in \ba_x$ satisfying $S(x)(u)=us$ and arbitrary $\si$ such that $\nabla^\si\in [\nabla^\si]$. In particular, $\un S$ and $S$ are smooth, because $\si$ is smooth. Therefore Claims (3) and (4) hold.

Since
\begin{gather*}S(x)\circ S(y)\circ S(x)^{-1}(S(x)(\si(u_0)))=S(x)(\si(u_0))s\end{gather*}
holds for $u_0$ in the f\/iber over $x$, the composition $S(x)\circ S(y)\circ S(x)^{-1}$ is an $s$-symmetry at the point $\un S(x)(y)$. The equality $S(x)\circ S(y)\circ S(x)^{-1}=S(\un S(x)(y))$ then follows from the uniqueness of $s$-symmetries. Therefore Claim (5) holds.

In particular, $\un S(x)\circ \un S(y)(y)=\un S(\un S(x)(y)) \circ \un S(x)(y)$ holds. This implies that
\begin{gather*}
\si'(v_0)\exp\big(\big(\U^{\si',S(x)}\big)_{\si'}(v_0s)\big) =(S(x)\circ S(y))^*\si'(v_0)=(S(\un S(x)(y))\circ S(x))^*\si'(v_0)\\
\hphantom{\si'(v_0)\exp\big(\big(\U^{\si',S(x)}\big)_{\si'}(v_0s)\big)}{} =\si'(v_0)\exp((\U^{\si',S(x)})_{\si'}(v_0))
\end{gather*}
holds for $v_0$ in the f\/iber over $y$ for arbitrary $\si'$ such that $\nabla^{\si'}\in [\nabla^\si]$. Thus
\begin{gather*}\Ad(s)\big(\U^{\si,S(x)}\big)_\si(v_0)=\big(\U^{\si,S(x)}\big)_\si(v_0)\end{gather*} holds and thus $[\nabla^\si]$ is an $\un S$-invariant class of Weyl connections.
\end{proof}

\section[Geometric properties of parabolic geometries of distinguished types and classif\/ication]{Geometric properties of parabolic geometries\\ of distinguished types and classif\/ication} \label{sec6}

In this section, we study properties of (locally) $s$-symmetric parabolic geometries of particular types $(G,P)$ for triples $(\fg,\fp,\mu)$ that are prolongation rigid outside of the $1$-eigenspace of~$s$ for~$\mu$ in the $1$-eigenspace of $s$. The properties follow from the position and shape of $\fg_-^s(1)$ inside of~$\fg_-$. We classify all triples $(\fg,\fp,\mu)$ where $\fg_-^s(1)$ has such a position and shape for generic~$s$.

\subsection{Classif\/ication results and notation}\label{sec61}

Let us use the characterization from Section \ref{sec4.2} for the classif\/ication of the triples $(\fg,\fp,\mu)$ that are prolongation rigid outside of the $1$-eigenspace of $s\in Z(G_0)$ such that $\mu$ is in the $1$-eigenspace of~$s$. We separate the classif\/ication into the series of Tables~\ref{ahs1}--\ref{ahs13}.
The main reason for such a separation is that parabolic geometries from dif\/ferent tables have dif\/ferent geometric properties and we divide the tables according to these properties.

\begin{Theorem} \label{t3.5}
Let $(\fg,\fp,\mu)$ be a triple obtained from one of Tables~{\rm \ref{ahs1}--\ref{ahs13}} in the following way:
\begin{itemize}\itemsep=0pt
\item The Lie algebra $\fg$ is a simple Lie algebra of the $($complex$)$ rank $n$ that is at least $A_4$, $B_4$, $C_4$, $D_5$ or some explicit Lie algebra of lower rank from the column~$\fg$.
\item The parabolic subalgebra $\fp$ is the parabolic algebra from {\rm \cite[Section 3.2.9]{parabook}} for the set $\Sigma$ in the column $\Sigma$.
\item The component of the harmonic curvature $\mu$ is specified by an ordered pair of simple roots of $\fg$ from the column $\mu$ that provides the highest weight of $\mu$ by the affine action of corresponding elements of the Weyl group, see {\rm \cite[Theorem 3.3.5]{parabook}}.
\item The component $\mu$ is contained in the $1$-eigenspace of $s$ for the elements $s\in Z(G_0)$ that have the eigenvalues $j_{i_a}$ from the columns $j_{i_a}$ on the irreducible $\fg_0$-components that are determined by the $i_a$th element of the set $\Sigma$.
\end{itemize}
Then $(\fg,\fp,\mu)$ is prolongation rigid outside of the $1$-eigenspace of $s$ if the eigenvalues $j_{i_a}$ of $s$ satisfy the condition in the column PR.

Tables~{\rm \ref{ahs1}--\ref{ahs13}} contain the complete classification of triples $(\fg,\fp,\mu)$ that are prolongation rigid outside of the $1$-eigenspace of~$s$ for~$\mu$ in the $1$-eigenspace of~$s$ $($except the cases that are conjugated by an outer automorphism of~$\fg$ to one of the listed entries$)$.
\end{Theorem}

The remaining notation we will use in the tables is the following:

We characterize the real form of $\fg$ by a number $q$ and a f\/ield $\{\mathbb{R, C, H}\}$.

The set $\Sigma$ characterizes the set of crossed nodes in the Dynkin or Satake diagram that provides the parabolic subalgebra~$\fp$. We use the ordering of nodes which is consistent with \cite[Appendix~B]{parabook} and we will not add the conjugated crossed nodes to~$\Sigma$ in the case of complex Lie algebras, $\mathfrak{su}(q,n+1-q)$ and $\mathfrak{so}(3,5)$. We distinguish the complex conjugated simple roots by~$'$.

If the column for the eigenvalue $j_{i_a}$ is blank, then the value of $j_{i_a}$ is generic. If the eigenvalue $j_{i_a}\notin \R$ and $\ln(j_{i_a}) = r_{i_a} +i\phi_{i_a}$, then either $r_{i_a}=0$ or $\phi_{i_a}=0$ and we specify only the non-zero one in the table.

If the column PR is missing or the condition is blank, then the triple $(\fg,\fp,\mu)$ is either prolongation rigid or the condition that~$\mu$ is contained in the $1$-eigenspace of $s$ is suf\/f\/icient for~$\mathfrak{a}_i$ corresponding to $\mu$ to be a subspace of $\fg_i^s(1)$ for all~$i$.

The classif\/ication tables are presented in the following subsections and the triples $(\fg,\fp,\mu)$ are obtained from the tables according to Theorem~\ref{t3.5}.

\subsection[Parabolic geometries with $\fg_-^s(1)=0$]{Parabolic geometries with $\boldsymbol{\fg_-^s(1)=0}$}\label{sec5.1}

Table \ref{ahs1} contains all triples $(\fg,\fp,\mu)$ with the property that if $s\in Z(G_0)$ is such that $(\fg,\fp,\mu)$ is prolongation rigid outside of the $1$-eigenspace of $s$, then $\fg_-^s(1)=0$. In particular, all AHS-structures that have a component of the harmonic curvature in the $1$-eigenspaces of some \smash{$s\in Z(G_0)$} are prolongation rigid outside of the $1$-eigenspace of $s$ and thus are contained in this table.

\begin{table}[th!]\caption{Theorem~\ref{6.2}}\label{ahs1}\centering \vspace{1mm}
\begin{tabular}{|c|c|c|c|}
\hline
$\fg$ & $\Sigma$ & $j_{i_1}$ & $\mu$\\
\hline
\hline
$\mathfrak{sl}(3,\mathbb{C})$ & $\{1\}$&$\phi_1$ & $(\alpha_1,\alpha_{1'})$ \\
\hline
$\mathfrak{sl}(3,\mathbb{C})$ & $\{1\}$&$\sqrt[3]{1}$& $(\alpha_{1'},\alpha_{2'})$ \\
\hline
$\mathfrak{sl}(4,\{\mathbb{R,C}\})$ & $\{1\}$&$\sqrt{1}$& $(\alpha_1,\alpha_2)$ \\
\hline
$\mathfrak{sl}(4,\mathbb{C})$ & $\{1\}$&$\phi_1$ & $(\alpha_1,\alpha_{1'})$ \\
\hline
$\mathfrak{sl}(4,\mathbb{C})$ & $\{1\}$&$\sqrt[3]{1}$& $(\alpha_{1'},\alpha_{2'})$ \\
\hline
$\mathfrak{sl}(n+1,\{\mathbb{R,C}\})$ & $\{1\}$&$\sqrt{1}$& $(\alpha_1,\alpha_2)$ \\
\hline
$\mathfrak{sl}(n+1,\mathbb{C})$ & $\{1\}$&$\phi_1$ & $(\alpha_1,\alpha_{1'})$ \\
\hline
$\mathfrak{sl}(n+1,\{\mathbb{R, C, H}\})$ &$\{2\}$ &$\sqrt{1}$& $(\alpha_2,\alpha_1)$ \\
\hline
$\mathfrak{sl}(n+1,\mathbb{C})$ & $\{p\}$&$\sqrt[3]{1}$& $(\alpha_{p'},\alpha_{p+1'})$ \\
\hline
$\mathfrak{so}(1,5),\mathfrak{so}(2,4),\mathfrak{so}(3,3),\mathfrak{so}(6,\mathbb{C})$, & & &\\
$\mathfrak{so}(1,6)$, $\mathfrak{so}(2,5)$, $\mathfrak{so}(3,4)$, $\mathfrak{so}(7,\mathbb{C})$ , &$\{1\}$ &$\sqrt{1}$& $(\alpha_1,\alpha_2)$ \\
$\mathfrak{so}(1,7)$, $\mathfrak{so}(2,6)$, $\mathfrak{so}(3,5)$, $\mathfrak{so}(4,4)$, $\mathfrak{so}(8,\mathbb{C})$& & & \\
\hline
$\mathfrak{so}(6,\mathbb{C})$, $\mathfrak{so}(7,\mathbb{C})$, $\mathfrak{so}(8,\mathbb{C})$ & $\{1\}$&$\sqrt[3]{1}$& $(\alpha_{1'},\alpha_{2'})$ \\
\hline
$\mathfrak{so}(7,\mathbb{C})$ & $\{3\}$&$\sqrt[3]{1}$ &$(\alpha_3,\alpha_2)$\\
\hline
$\mathfrak{so}(q,n-q)$, $\mathfrak{so}(n,\mathbb{C})$ & $\{1\}$&$\sqrt{1}$ & $(\alpha_1,\alpha_2)$ \\
\hline
$\mathfrak{so}(n,\mathbb{C})$ & $\{1\}$&$\sqrt[3]{1}$& $(\alpha_{1'},\alpha_{2'})$ \\
\hline
$\mathfrak{so}(2n,\mathbb{C})$ & $\{n\}$& $\sqrt[3]{1}$& $(\alpha_{n'},\alpha_{n-2'})$ \\
\hline
$\mathfrak{so}(2n+1,\mathbb{C})$ & $\{n\}$& $\sqrt[5]{1}$& $(\alpha_{n'},\alpha_{n-1'})$ \\
\hline
$\mathfrak{sp}(4,\mathbb{C})$ & $\{1\}$&$\sqrt[3]{1}$& $(\alpha_{1},\alpha_{2})$\\
\hline
$\mathfrak{sp}(4,\mathbb{C})$ & $\{1\}$&$\sqrt[3]{1}$& $(\alpha_{1'},\alpha_{2'})$\\
\hline
$\mathfrak{sp}(4,\mathbb{C})$ & $\{2\}$&$\sqrt[3]{1}$& $(\alpha_{2'},\alpha_{1'})$\\
\hline
$\mathfrak{sp}(6,\mathbb{C})$ & $\{2\}$&$\sqrt[5]{1}$& $(\alpha_{2'},\alpha_{3'})$\\
\hline
$\mathfrak{sp}(6,\mathbb{C})$ & $\{3\}$&$\sqrt[3]{1}$& $(\alpha_{3'},\alpha_{2'})$\\
\hline
$\mathfrak{sp}(2n,\mathbb{C})$ & $\{n-1\}$&$\sqrt[5]{1}$& $(\alpha_{n-1'},\alpha_{n'})$\\
\hline
$\mathfrak{sp}(2n,\mathbb{C})$ & $\{n\}$&$\sqrt[3]{1}$& $(\alpha_{n'},\alpha_{n-1'})$\\
\hline
$\mathfrak{e}_{6}(\mathbb{C})$ & $\{1\}$&$\sqrt[3]{1}$&$(\alpha_{1'},\alpha_{2'})$\\
\hline
$\mathfrak{e}_{7}(\mathbb{C})$ & $\{1\}$&$\sqrt[3]{1}$&$(\alpha_{1'},\alpha_{2'})$\\
\hline
\end{tabular}
\end{table}

\begin{Example}
Before we formulate the general result, let us demonstrate how the results for (locally) symmetric conformal geometries that we presented in~\cite{GZ4} can be obtained from Table~\ref{ahs1} and Theorem~\ref{6.2}:

There are rows with $\fg=\mathfrak{so}(q,n-q)$ and $\Sigma=\{1\}$ in Table~\ref{ahs1} and the triples $(\mathfrak{so}(q,n-q)$, $\fp_{\{1\}},\mu_{(\alpha_1,\alpha_2)})$ are prolongation rigid outside of the $1$-eigenspace of $s$ for $n>5$ and $q>0$. We read of the corresponding line that the eigenvalue $j_{i_1}=\sqrt{1}$. Thus $s=m$ and the $m$-symmetries in question are the symmetries of conformal geometries presented in~\cite{GZ4}. We get immediately from Theorem~\ref{6.2} that Theorem~\ref{t1.1} holds for conformal geometries.
\end{Example}

In the following theorem, we summarize geometric properties of geometries from Table \ref{ahs1} and prove the last claim of Theorem \ref{t1.3}.

\begin{Theorem} \label{6.2}
Assume $(\fg,\fp,\mu)$ is prolongation rigid outside of the $1$-eigenspace of $s$ for $s\in Z(G_0)$ such that $\fg_-^s(1)=0$ holds. If the harmonic curvature $\kappa_H$ of the $($locally$)$ $s$-symmetric parabolic geometry $(\ba\to M,\om)$ of type $(G,P)$ has a non-zero component in~$\mu$ at some~$x$, then:
\begin{enumerate}\itemsep=0pt
\item[$1.$] The parabolic geometry is $($locally$)$ homogeneous, $\kappa_H(x)\neq 0$ at all $x\in M$ and there is a~unique smooth system of $($local$)$ $s$-symmetries~$S$ on~$M$.
\item[$2.$] There is a unique distinguished Weyl structure $\si$ which is uniquely characterized by one of the following equivalent properties:
\begin{enumerate}\itemsep=0pt
\item[$(a)$] The equalities $\nabla^\si T^\si=0$, $s.(T^\si)_\si=(T^\si)_\si$, $\nabla^\si R^\si=0$ and $s.(R^\si)_\si=(R^\si)_\si$ hold for the torsion and the curvature of the Weyl connection $\nabla^\si$.
\item[$(b)$] The Weyl connection $\nabla^\si$ is $\un S$-invariant.
\item[$(c)$] All $($local$)$ automorphisms of the parabolic geometry cover affine transformations of~$\nabla^\si$.
\item[$(d)$] All $($local$)$ diffeomorphisms $s_x^\si$ are affine transformations of~$\nabla^\si$.
\item[$(e)$] All $($local$)$ $P$-bundle morphisms $s_{\si(u_0)}$ are $($local$)$ $s$-symmetries.
\end{enumerate}
\item[$3.$] The pseudo-group generated by all local $s$-symmetries is transitive on $M$ and its connected component of identity is generated by the flows of the Lie algebra~$\fl$, which is the vector subspace of $\fg_-\oplus \fg_0$, generated by~$\fg_-$ by the bracket $(T^\si+R^\si)_\si$ on $\wedge^2 \fg_-^*\otimes \fl$ and the natural bracket on the rest of~$\fl$.
\item[$4.$] The equalities \begin{gather*}\un S(x)=\un s_{\si(u_0)}=s_x^\si\end{gather*}
hold for the Weyl structure $\si$ from Claim~$(2)$. In particular,
\begin{itemize}\itemsep=0pt
\item the maps $\un S(x)$ can be extended to a larger neighbourhood of $x$ as long as the corresponding geodesic transformations of $\nabla^\si$ are defined,
\item $\un S(x)\circ \un S(y)\circ \un S(x)^{-1}(z)=\un S(\un S(x)(y))(z)$ holds for $(x,y,z)$ in some neighbourhood of the diagonal in $M\times M\times M$, and
\item for each eigenvalue $a$, the distribution $TM^s(a)$ is preserved by all $($local$)$ automorphisms of the parabolic geometry.
\end{itemize}
\end{enumerate}
\end{Theorem}
\begin{proof}
Let $U\subset M$ be the set of points $x$ such that $\kappa_H(x)$ has a non-zero component in $\mu$. Then there is a unique system of (local) $s$-symmetries on $U$ due to the prolongation rigidity of the triple $(\fg,\fp,\mu)$ outside of the $1$-eigenspace of~$s$. It suf\/f\/ices to prove the theorem under the assumption $U=M$, because if we prove Claim~(3) on~$U$, then the equality $U=M$ follows from the (local) homogeneity, i.e., Claim~(1) follows from Claim~(3). Then Claim~(4) follows from Claim~(2) due to Claims~(5) and~(7) of Proposition~\ref{almost-ws}.

Therefore, it suf\/f\/ices to prove Claims (2) and (3) under the assumption $U=M$ to complete the proof. If $\fg^s_-(1)=0$, then $\fp^s_+(1)=0$ and Proposition~\ref{almost-ws} implies that there is a~unique $S$-invariant Weyl structu\-re~$\si$. It follows from Propositions~\ref{auto-action} and~\ref{geodesic-s-symm} that the Weyl structure~$\si$ satisf\/ies (2b) if and only if it satisf\/ies (2e). Further, Proposition~\ref{geodesic-s-symm} and Corollary~\ref{c3.3} imply that the Weyl structure~$\si$ satisf\/ies (2e) if and only if it satisf\/ies (2d).

We show now that (2b) implies (2a). The torsion and the curvature of $\un S$-invariant Weyl connection $\nabla^\si$ are $\un S$-invariant. In particular, \begin{gather*}s.(T^\si)_\si(u_0)=(\un S(p_0(u_0))^*T^\si(p_0(u_0)))_\si(u_0)=(T^\si)_\si(u_0)\end{gather*} and \begin{gather*}s.(R^\si)_\si(u_0)=(\un S(p_0(u_0))^*R^\si(p_0(u_0)))_\si(u_0)=(R^\si)_\si(u_0)\end{gather*} hold for all $u_0\in \ba_0$ for the natural action $.$ of $G_0$ on the values of $(T^\si)_\si$ and $(R^\si)_\si$. Since the same arguments can be applied on $\nabla^\si T^\si$ and $\nabla^\si R^\si$, it follows that $(\nabla^\si_\xi T^\si)_\si=s.(\nabla^\si_\xi T^\si)_\si=a(\nabla^\si_\xi T^\si)_\si$ and $(\nabla^\si_\xi R^\si)_\si=s.(\nabla^\si_\xi R^\si)_\si=a(\nabla^\si_\xi R^\si)_\si$ hold for any vector f\/ield $\xi$ on $M$ such that $(\xi)_\si(u_0) \in \fg^s_-(a^{-1})$ for all $u_0\in \ba_0$. Thus (2b) implies (2a), because $\fg^s_-(1)=0$.

Claim (2a) implies that $\nabla^\si$ is a locally af\/f\/inely homogeneous connection. Therefore, according to \cite[Section 1.5]{disertace}, the af\/f\/ine geometry $(M,\nabla^\si)$ can be encoded as a locally homogeneous Cartan geometry of type $(\fg_-\rtimes {\rm Gl}(\fg_-),{\rm Gl}(\fg_-))$ on the f\/irst-order frame bundle $\mathcal{P}^1M$. Moreover, the assumptions of \cite[Lemma~2.2]{GZ2} are satisf\/ied, because $(T^\si+R^\si)_\si(\iota_\si(u_0))$ is the bracket of the inf\/initesimal af\/f\/ine transformation at $\iota_\si(u_0)\in \iota_\si(\ba_0)\subset \mathcal{P}^1M$. Thus there is a (local) af\/f\/ine transformation $A$ of $\nabla^\si$ such that $(A)_\si(u_0)=s$. Therefore Claim (2d) follows from Claim (2a) and Proposition \ref{geodesic-s-symm} due to the uniqueness of $s$-symmetries. In particular, if we consider a (local) one-parameter subgroup $\exp(t\xi)$ for an inf\/initesimal af\/f\/ine transformation $\xi$, then $\exp(t\xi)\un S(x)\exp(-t\xi)$ is the (local) $s$-symmetry at $\exp(t\xi)(x)$ and the map $\frac{d}{dt}|_{t=0}\exp(t\xi)\un S(x)\exp(-t\xi)\un S(x)^{-1}$ maps $\xi$ into $\fl$. If $(\mathcal{P}^1\xi)_\si(u_0)=X$, then the element $X-\Ad(s)(X)$ is contained in $\fl$. Thus $\fg_-\subset \fl$ as a vector subspace. Thus the f\/lows of the Lie algebra $\fl$ generate
a sub-pseudo-group, which is the connected component of identity of the pseudo-group generated by local $s$-symmetries. Since $\Ad(s)$ preserves $\fl$, Claim (3) follows.

We can use the results from \cite[Theorem 1.3]{GZ3} due to the local homogeneity and (2c) follows from (2a). Clearly (2c) implies (2d), which completes the proof.
\end{proof}

\subsection[Parabolic geometries with distinguished parabolic subalgebras $\fg_-^s(1)+\fp$]{Parabolic geometries with distinguished parabolic subalgebras $\boldsymbol{\fg_-^s(1)+\fp}$}\label{sec5.2}

There are triples $(\fg,\fp,\mu)$ that are prolongation rigid outside of the $1$-eigenspace of $s$ which admit $1$-eigen\-space in~$\fg_-$ for some $s$ such that $\fv:=\fg_-^s(1)+\fp$ is a parabolic subalgebra of $\fg$ such that the harmonic curvature in $\mu$ vanishes on insertions of elements of~$\fv/\fp$ at all points of~$M$. These are listed in Tables~\ref{ahs2},~\ref{ahs3} and~\ref{ahs4} due to \cite[Propositions~6.2 and~A.2]{GZ3}.

\begin{table}[th!]\centering \caption{Theorem \ref{6.3}, part with $|\Si|=2$.}\label{ahs2} \vspace{1mm}
\begin{tabular}{|c|c|c|c|c|c|}
\hline
$\fg$ & $\Sigma$ & $j_{i_1}$& $j_{i_2}$& $\mu$&PR \\
\hline
\hline
$\mathfrak{sl}(3,\mathbb{C})$ &$\{1,2\}$&& $2r_1$& $(\alpha_{1},\alpha_{1'})$& \\
\hline
$\mathfrak{sl}(3,\mathbb{C})$& $\{1,2\}$ &$-\frac{2}3\phi_2$& & $ (\alpha_{1'},\alpha_{2'})$& $r_2=0$, $\phi_2=2\pi$ \\
\hline
$\mathfrak{sl}(3,\mathbb{C})$& $\{1,2\}$ && $-\frac23\phi_1$& $ (\alpha_{2'},\alpha_{1'})$& $r_1=0$, $\phi_1=2\pi$ \\
\hline
$\mathfrak{sl}(4,\{\R,\C\})$ & $\{1,2\}$ & $j_2^{-2}$ & & $(\alpha_2,\alpha_1)$& $j_2=-1$ \\
\hline
$\mathfrak{sl}(4,\R)$ & $\{1,2\}$ &$j_2^2$ && $(\alpha_2,\alpha_3)$ & \\
\hline
$\mathfrak{sl}(4,\mathbb{C})$ &$\{1,2\}$&& $2r_1$& $(\alpha_{1},\alpha_{1'})$& \\
\hline
$\mathfrak{sl}(4,\mathbb{C})$& $\{1,2\}$ &$-\frac23\phi_2$& & $ (\alpha_{1'},\alpha_{2'})$& $r_2=0$, $\phi_2=2\pi$ \\
\hline
$\mathfrak{sl}(4,\mathbb{C})$& $\{1,2\}$ && $-\frac23\phi_1$& $ (\alpha_{2'},\alpha_{1'})$& $r_1=0$, $\phi_1=2\pi$ \\
\hline
$\mathfrak{sl}(4,\R)$ & $\{1,3\}$ &&$j_1^2$ &$(\alpha_1,\alpha_2)$& \\
\hline
$\mathfrak{sl}(4,\mathbb{C})$ &$\{1,3\}$&& $2r_1$& $(\alpha_{1},\alpha_{1'})$& \\
\hline
$\mathfrak{sl}(n+1,\{\R,\C\})$ & $\{1,2\}$ & $j_2^{-2}$ & & $(\alpha_2,\alpha_1)$& $j_2=-1$ \\
\hline
$\mathfrak{sl}(n+1,\mathbb{C})$ &$\{1,2\}$&& $2r_1$& $(\alpha_{1},\alpha_{1'})$& \\
\hline
$\mathfrak{sl}(n+1,\R)$ & $\{1,3\}$ &&$j_1^2$ &$(\alpha_1,\alpha_2)$& \\
\hline
$\mathfrak{sl}(n+1,\mathbb{C})$ &$\{1,p\}$&& $2r_1$& $(\alpha_{1},\alpha_{1'})$&$r_1=0$\\
& $2<p<n$& & &&\\
\hline
$\mathfrak{sl}(n+1,\{\R,\C\})$ & $\{1,p\}$ &&$j_1^2$ &$(\alpha_1,\alpha_2)$& $j_1=-1$ \\
& $3<p<n$& & &&\\
\hline
$\mathfrak{sl}(n+1,\R)$ & $\{1,n\}$ &&$j_1^2$ &$(\alpha_1,\alpha_2)$& \\
\hline
$\mathfrak{sl}(n+1,\mathbb{C})$ &$\{1,n\}$&& $2r_1$& $(\alpha_{1},\alpha_{1'})$& \\
\hline
$\mathfrak{sl}(n+1,\R)$ & $\{2,3\}$ &&$j_2^2$ & $(\alpha_2,\alpha_1)$ &\\
\hline
$\mathfrak{sl}(n+1,\{\R,\C,\mathbb{H}\})$ & $\{2,p\}$ &&$j_2^2$ & $(\alpha_2,\alpha_1)$ & $j_2=-1$\\
& $3<p<n$& & &&\\
\hline
$\mathfrak{sl}(n+1,\{\R,\mathbb{H}\})$ & $\{2,n\}$ &&$j_2^2$ & $(\alpha_2,\alpha_1)$ & \\
\hline
$\mathfrak{sl}(n+1,\mathbb{C})$& $\{p,p+1\}$ && $-\frac23\phi_p$& $ (\alpha_{p+1'},\alpha_{p'})$& $r_p=0$, $\phi_p=2\pi$ \\

\hline
$\mathfrak{so}(2,5),\mathfrak{so}(3,4),\mathfrak{so}(7,\C),$&&&& & \\
$\mathfrak{so}(2,6),\mathfrak{so}(3,5),$&$\{ 1,2\}$&$\sqrt{1} $&& $(\alpha_1,\alpha_2)$ &$j_2=1$ \\
$\mathfrak{so}(4,4),\mathfrak{so}(8,\C)$& & & & &\\
\hline
$\mathfrak{so}(4,4)$&$\{ 1,4\}$& & $j_1^2$& $(\alpha_1,\alpha_2)$ & \\
\hline
$\mathfrak{so}(q,n-q),\mathfrak{so}(n,\C)$&$\{ 1,2\}$&$\sqrt{1} $&& $(\alpha_1,\alpha_2)$ &$j_2=1$ \\
\hline
$\mathfrak{so}(n,n), \mathfrak{so}(2n,\C)$&$\{ 1,n\}$& & $j_1^2$& $(\alpha_1,\alpha_2)$ &$j_1=-1$ \\
\hline
\end{tabular}
\end{table}

\begin{table}[th!]\centering \caption{Theorem \ref{6.3}, part with $|\Si|=3$.}\label{ahs3}\vspace{1mm}

\begin{tabular}{|c|c|c|c|c|c|c|}
\hline
$\fg$ & $\Sigma$ & $j_{i_1}$& $j_{i_2}$& $j_{i_3}$&$\mu$&PR \\
\hline
\hline
$\mathfrak{sl}(4,\mathbb{R})$ &$\{1,2,3\}$& & & $j_1j_2^2 $ &$(\alpha_2,\alpha_1)$& $j_1=1$\\
\hline
$\mathfrak{sl}(n+1,\mathbb{R})$ &$\{1,2,3\}$& & & $j_1j_2^2 $ &$(\alpha_2,\alpha_1)$& $j_1=1$ \\
\hline
$\mathfrak{sl}(n+1,\{\mathbb{R},\C\})$ &$\{1,2,p\}$& & & $j_1j_2^2 $ &$(\alpha_2,\alpha_1)$& $j_1=1$, $j_2=-1$ \\
& $3<p<n$& & &&&\\
\hline
$\mathfrak{sl}(n+1,\mathbb{R})$ &$\{1,2,n\}$& & & $j_1j_2^2 $ &$(\alpha_2,\alpha_1)$& $j_1=1$ \\
\hline
$\mathfrak{so}(4,4)$& $\{1,2,4\}$&&& $j_1^2 $&$(\alpha_1,\alpha_2)$& $j_2=1$\\
\hline
$\mathfrak{so}(n,n)$, $\mathfrak{so}(4n,\mathbb{C})$& $\{1,2,n\}$&&& $j_1^2 $&$(\alpha_1,\alpha_2)$& $j_1=-1$, $j_2=1$\\
\hline
\end{tabular}
\end{table}

\begin{table}[th!]\centering\caption{Theorem \ref{6.3}, part with $|\Si|=4$.}\label{ahs4}\vspace{1mm}
\begin{tabular}{|c|c|c|c|c|c|c|c|c|}
\hline
$\fg$ & $\Sigma$ & eigenvalues& $\mu$& PR \\
\hline
\hline
$\mathfrak{sl}(n+1,\mathbb{R})$ & $\{1,2,3,p\}$ &$j_p=j_1j_2^2j_3^{-1}$&$(\alpha_2,\alpha_1)$& $j_1=1$, $j_3=j_2^2$\\
& $3<p<n$& & &\\
\hline
$\mathfrak{sl}(n+1,\{\mathbb{R},\C\})$ & $\{1,2,p,q\}$&$j_q=j_1j_2^2j_p^{-1}$&$(\alpha_2,\alpha_1)$& $j_1=1$, $j_2=-1$, $j_p=1$\\
& $3<p,q<n$& & &\\
\hline
$\mathfrak{sl}(n+1,\mathbb{R})$ & $\{1,2,p,n\}$&$j_n=j_1j_2^2j_p^{-1}$&$(\alpha_2,\alpha_1)$& $j_1=1$, $j_p=1$\\
& $3<p<n$& & &\\
\hline
\end{tabular}
\end{table}

\begin{Example}\label{expg}To demonstrate our results, let us look in Table \ref{ahs2} on the row $\fg=\mathfrak{sl}(n+1,\R)$ and $\Sigma=\{1,2\}$ which corresponds to generalized path-geometries (for systems of second-order ODEs), see \cite[Sections 4.4.3--4.4.5]{parabook}. These parabolic geometries generally have two harmonic curvatures, one torsion $\kappa_{(\alpha_1,\alpha_2)}$ and one curvature $\kappa_{(\alpha_2,\alpha_1)}$. However, they fall in Table \ref{ahs2} only when the torsion $\kappa_{(\alpha_1,\alpha_2)}$ vanishes and the harmonic curvature consists only of the curva\-tu\-re~$\kappa_{(\alpha_2,\alpha_1)}$ corresponding to $\mu_{(\alpha_2,\alpha_1)}$. There are many $s\in Z(G_0)$ that act trivially on $\mu_{(\alpha_2,\alpha_1)}$, but the triple $(\mathfrak{sl}(n+1,\R),\fp_{\{1,2\}},\mu_{(\alpha_2,\alpha_1)})$ is prolongation rigid outside of the $1$-eigenspace of~$s$ only for~$s\in Z(G_0)$ with eigenvalues $j_1=1$, $j_2=-1$. In such case, $\fv=\fp_{\{2\}}$ is the parabolic subalgebra of~$\fg$ corresponding to $\Sigma=\{2\}$.

The torsion-freeness of generalized path-geometries implies that the space of local solutions of the corresponding ODEs carries a Grassmanian structure, which is a parabolic geometry on the local leaf space of type $(G,Q)$ from Theorem~\ref{6.3}.
Therefore if $(\ba\to M,\om)$ is a (locally) $s$-symmetric torsion-free generalized path-geometry with a non-zero harmonic curvature, then we conclude from Theorem \ref{6.3} that the space of local solutions~$N$ is a locally symmetric space~$(N,\un S)$, while~$M$ together with the system of (local) $s$-symmetries $S$ is a~ref\/lexion space~$(M,\un S)$ over~$(N,\un S)$, see~\cite{G1}. Let us emphasize that due to dimensional reasons and the formula \cite[Theorem~5.2.9]{parabook}, the pseudo-group generated by all local $s$-symmetries is locally transitive at $x\in M$ if and only if the Rho-tensor $\Rho^\si(n(x))$ of the $S$-invariant Weyl structure~$\si$ on~$N$ does not vanish on~$T_{n(x)}N$.
\end{Example}

We summarize geometric properties of geometries from Tables~\ref{ahs2}, \ref{ahs3} and~\ref{ahs4} in the following theorem.

\begin{Theorem}\label{6.3}
Assume $(\fg,\fp,\mu)$ is prolongation rigid outside of the $1$-eigenspace of $s$ for $s\in Z(G_0)$ such that $\fv=\fg_-^s(1)+\fp$ is a parabolic subalgebra of $\fg$ and $\fv/\fp$ inserts trivially into the harmonic curvature $\kappa_H$ of the $($locally$)$ $s$-symmetric parabolic geometry $(\ba\to M,\om)$ of type $(G,P)$. If $\kappa_H$ has a non-zero component in $\mu$ at some $x$, then:
\begin{enumerate}\itemsep=0pt
\item[$1.$] The inequality $\kappa_H\neq 0$ holds in an open dense subset of $M$, and there is a unique smooth system of $($local$)$ $s$-symmetries $S$ on $M$.
\item[$2.$] There are
\begin{itemize}\itemsep=0pt
\item a parabolic subgroup $Q$ of $G$ with the Lie algebra $\fv$ such that $P\subset Q$,
\item a neighbourhood $U_x$ of each $x\in M$ with the local leaf space $n\colon U_x\to N$ for the foliation given by the integrable distribution $Tp\circ \om^{-1}(\fv)$, and
\item a $($locally$)$ $s$-symmetric parabolic geometry $(\ba'\to N,\om')$ of type~$(G,Q)$ satisfying the assumptions of Theorem~{\rm \ref{6.2}},
\end{itemize}
such that $(\ba|_{U_x}\to U_x,\om|_{U_x})$ is isomorphic to an open subset of $(\ba'\to \ba'/P,\om')$ for each~$x$. In particular, there is a unique $s$-symmetry $\un S(n(y))$ on $(\ba'\to N,\om')$ at each $n(y)\in N$ such that $n\circ \un S(y)=\un S(n(y))\circ n$ holds for all $y\in U_x$ in the fiber over $n(y)$.
\item[$3.$] The connected component of identity of the pseudo-group generated by all local $s$-sym\-met\-ries is generated by the flows of the Lie algebra $\fl$, which is the vector subspace of $\fv^{\rm op}$, generated by $\fv^{\rm op}_+$ by the bracket $(T^\si(n(x))+R^\si(n(x)))_\si$ on $\wedge^2 (\fv^{\rm op}_+)^*\otimes \fl$ and the natural bracket on the rest of $\fl$ for the $S$-invariant Weyl structure $\si$ on $(\ba'\to N,\om')$, where $\fv^{\rm op}$ is the opposite parabolic subalgebra of $\fg$ to $\fv$.

The pseudo-group generated by all local $s$-symmetries is locally transitive at $x$ if and only if $\fv/\fp\subset \fl/(\fl\cap \fp)$, i.e., if and only if $(R^\si(n(x)))_\si$ spans the whole $\fv/\fp$.
\item[$4.$] There is a bijection between
\begin{itemize}\itemsep=0pt
\item the almost $S$-invariant Weyl structures on $U_x$, and
\item the reductions of the image in $\ba'$ of the $($unique$)$ $S$-invariant Weyl structure~$\si$ on~$N$ $($that exists due to Theorem~{\rm \ref{6.2})} to $\exp(\fg_-^s(1))\rtimes G_0\subset Q_0$.
\end{itemize}
A reduction corresponds to an $S$-invariant Weyl structure on $U_x$ if and only if it is a~holonomy reduction of~$\nabla^\si$.
\item[$5.$] In particular,
\begin{itemize}\itemsep=0pt
\item the maps $\un S(x)$ can be extended to a larger neighbourhood of~$x$ as long as the corresponding geodesic transformations of $\nabla^\si$ on $N$ are defined,
\item $\un S(x)\circ \un S(y)\circ \un S(x)^{-1}(z)=\un S({\un S(x)(y)})(z)$ holds for $(x,y,z)$ in some neighbourhood of the diagonal in $M\times M\times M$,
\item the distribution $TM^s(1)$ is the vertical distribution of the local leaf space $n\colon U_x\to N$,
\item for each eigenvalue $a$, $Tn(T_xM^s(a))$ is the $a$-eigenspace of $T_{n(x)}\un S(n(x))$ in $T_{n(x)}N$, and
\item for each eigenvalue $a$, the distribution $TM^s(a)$ is preserved by all $($local$)$ automorphisms of the parabolic geometry.
\end{itemize}
\end{enumerate}
\end{Theorem}
\begin{proof}
Claim (1) is a direct consequence of Claims (2) and (3), because $\kappa_H\neq 0$ holds for the harmonic curvature of $(\ba'\to N,\om')$ and thus $\kappa_H=0$ can hold only in the subset of the f\/iber corresponding to a (Zariski) closed subset of $Q$. Claim (2) follows from \cite[Theorem~3.3]{cap-cor} and the fact that $(\fv^{\rm op}_+)^s(1)=0$. Then Claim (3) is a clear consequence of Theorem~\ref{6.2}. Claim~(4) follows from the comparison of images in $\ba'$ of the $S$-invariant Weyl structure on $N$ and the almost $S$-invariant Weyl structures on $U_x$, because they intersect precisely in a reduction to $\exp(\fg_-^s(1))\rtimes G_0\subset Q_0$, i.e., in a subbundle with the structure group $\exp(\fg_-^s(1))\rtimes G_0$. Claim~(5) is a consequence of Claim (4) of Theorem \ref{6.2} and Claim (2).
\end{proof}

\subsection[Parabolic geometries with $\fg_{-1}^s(1)=0$]{Parabolic geometries with $\boldsymbol{\fg_{-1}^s(1)=0}$}\label{sec5.3}
There are triples $(\fg,\fp,\mu)$ that are prolongation rigid outside of the $1$-eigenspace of $s$ which admit a $1$-eigenspace in $\fg_-$ for some $s$ such that $\fg_{-1}^s(1)=0$ holds, but which do not generically satisfy $\fg_-^s(1)=0$. These are contained in Tables~\ref{ahs5} and~\ref{ahs6}.

\begin{table}[th!]\centering\caption{Theorem \ref{6.4}, part with $|\Si|=1$.}\label{ahs5}\vspace{1mm}
\begin{tabular}{|c|c|c|c|}
\hline
$\fg$ & $\Sigma$ & $j_{i_1}$ & $\mu$\\
\hline
\hline
$\mathfrak{su}(1,2)$ & $\{1\}$& $\sqrt[4]{1}$ & $(\alpha_1,\alpha_2)$ \\
\hline
$\mathfrak{su}(1,3)$, $\mathfrak{su}(2,2)$ & $\{1\}$& $\phi_1$ & $(\alpha_1,\alpha_3)$ \\
\hline
$\mathfrak{su}(1,3)$, $\mathfrak{su}(2,2)$ & $\{1\}$&$\sqrt[3]{1}$ & $(\alpha_1,\alpha_2)$ \\
\hline
$\mathfrak{su}(q,n+1-q)$ & $\{1\}$& $\phi_1$ & $(\alpha_1,\alpha_n)$ \\
\hline
$\mathfrak{su}(q,n+1-q)$ & $\{1\}$&$\sqrt[3]{1}$ & $(\alpha_1,\alpha_2)$ \\
\hline
$\mathfrak{su}(q,n+1-q)$ & $\{2\}$ &$\sqrt[3]{1}$& $(\alpha_2,\alpha_1)$\\
\hline
$\mathfrak{so}(3,5)$&$\{ 3\}$&$\sqrt[3]{1}$&$(\alpha_3,\alpha_2)$ \\
\hline
$\mathfrak{sp}(4,\mathbb{C})$ & $\{1\}$&$\phi_1$& $(\alpha_{1},\alpha_{1'})$\\
\hline
$\mathfrak{sp}(6,\{\mathbb{R,C}\})$ & $\{1\}$ &$\sqrt{1}$& $(\alpha_1,\alpha_2)$ \\
\hline
$\mathfrak{sp}(6,\mathbb{C})$ & $\{1\}$&$\phi_1$& $(\alpha_1,\alpha_1')$ \\
\hline
$\mathfrak{sp}(1,2), \mathfrak{sp}(6,\{\mathbb{R, C}\})$ & $\{2\}$& $\sqrt{1}$ & $(\alpha_2,\alpha_1)$ \\
\hline
$\mathfrak{sp}(2n,\{\mathbb{R,C}\})$ & $\{1\}$ &$\sqrt{1}$& $(\alpha_1,\alpha_2)$ \\
\hline
$\mathfrak{sp}(2n,\mathbb{C})$ & $\{1\}$&$\phi_1$& $(\alpha_1,\alpha_1')$ \\
\hline
$\mathfrak{sp}(q,n-q), \mathfrak{sp}(2n,\{\mathbb{R, C}\})$ & $\{2\}$& $\sqrt{1}$ & $(\alpha_2,\alpha_1)$ \\
\hline
$\mathfrak{g}_{2}(\{2,\mathbb{C}\})$ & $\{1\}$&$\sqrt[4]{1}$&$(\alpha_1,\alpha_2)$\\
\hline
\end{tabular}
\end{table}

\begin{table}[th!]\centering\caption{Theorem \ref{6.4}, part with $|\Si|=2$.}\label{ahs6}\vspace{1mm}

\begin{tabular}{|c|c|c|c|c|}
\hline
$\fg$ & $\Sigma$ & $j_{i_1}$& $j_{i_2}$& $\mu$ \\
\hline
\hline
$\mathfrak{sl}(3,\{\R,\C\})$ & $\{1,2\}$ & $\sqrt[4]{1}$ & $\sqrt[4]{1}^{3}$& $(\alpha_1,\alpha_2)$ \\
\hline
$\mathfrak{sl}(3,\C)$ & $\{1,2\}$ & $\sqrt[5]{1}$& $\sqrt[5]{1}^{3}$& $(\alpha_1,\alpha_2)$ \\
\hline
$\mathfrak{sl}(4,\{\R,\C\})$ & $\{1,3\}$ & & $j_1^{-1}$& $(\alpha_1,\alpha_3)$ \\
\hline
$\mathfrak{sl}(n+1,\{\R,\C\})$ & $\{1,n\}$ & & $j_1^{-1}$& $(\alpha_1,\alpha_n)$ \\
\hline
$\mathfrak{so}(2,3), \mathfrak{so}(5,\mathbb{C})$ & $\{1,2\}$&$\sqrt[4]{1}$ & $\sqrt[4]{1}^{3}$&$(\alpha_{1},\alpha_2)$\\
\hline
$\mathfrak{so}(5,\mathbb{C})$ & $\{1,2\}$&$j_1^5= 1$ or $j_1^7= 1$&$j_1^3$&$(\alpha_{1},\alpha_2)$\\
\hline
$\mathfrak{so}(3,4)$&$\{ 1,3\}$&$j_3^3$&& $(\alpha_3,\alpha_2)$ \\
\hline
\end{tabular}
\end{table}

\begin{Example}We see that partially integrable almost CR-structures of hypersurface type are contained in Table~\ref{ahs5}, i.e., $\fg=\mathfrak{su}(q,n+1-q)$, $q>0$, $n>1$ and $\Sigma=\{1\}$. With the exception of the case $n=2$, there are two possible components of the harmonic curvature such that the triple $(\mathfrak{su}(q,n+1-q),\fp_{\{1\}},\mu)$ is prolongation rigid outside of the $1$-eigenspace of $s$ for $s\in Z(G_0)$ with the specif\/ied eigenvalue. Moreover, $\fg_-^s(1)=\fg_{-2}$ holds in all the cases when $(\mathfrak{su}(q,n+1-q),\fp_{\{1\}},\mu)$ is prolongation rigid outside of the $1$-eigenspace of $s$. Let us emphasize that the possibility $s^3=\id$ is available for both components of the harmonic curvature. Since $\fg_-^s(1)=\fg_{-2}$, we need some additional assumptions in Theorem \ref{6.4} to show that $(M,S)$ is (locally, under these assumptions) either a (locally) homogeneous one-dimensional f\/iber bundle over (reduced) $\mathbb{S}^1$-space, or a $\mathbb{Z}_3$-space or a symmetric space (due to \cite[Proposition~7.3]{GZ3}, see also~\cite{Loos2}) that carries some $\un S$-invariant Weyl connection on~$TM$. In particular, all such parabolic geometries can be classif\/ied using \cite[Theorem~5.1.4]{disertace} and Theorem~\ref{6.4}, if one knows the classif\/ication of $\mathbb{S}^1$-spaces, $\mathbb{Z}_3$-spaces and symmetric spaces. Let us emphasize that a part of the classif\/ication is done in~\cite{G2}.
\end{Example}

As mentioned in the example, we need an additional assumption on where the local $s$-symmetries are def\/ined for parabolic geometries in question.

\begin{Theorem} \label{6.4}
Let $(\fg,\fp,\mu)$ be prolongation rigid outside of the $1$-eigenspace of $s$ for $s\in Z(G_0)$ such that $\fg_{-1}^s(1)=0$ holds. Assume that for the $($locally$)$ $s$-symmetric parabolic geometry $(\ba\to M,\om)$ of type~$(G,P)$, the open subset~$U$ of~$M$ containing the points at which $\kappa_H$ has a non-zero component in~$\mu$ is non-trivial, and the maps $\un S(x)(y)$ and $\un S(x)\circ \un S(y)^{-1}(z)$ are defined on neighbourhoods of diagonals in $U\times U$ and $U\times U\times U$ for the unique system $S$ of $($local$)$ $s$-symmetries on $U$. Then:
\begin{enumerate}\itemsep=0pt
\item[$1.$]
The parabolic geometry is $($locally$)$ homogeneous and $U=M$, i.e., $\kappa_H(x)\neq 0$ at all $x\in M$ and there is a unique smooth system of $($local$)$ $s$-symmetries $S$ on $M$.
\item[$2.$] There is a class of distinguished Weyl structures characterized by one of the following equivalent properties for each Weyl structure $\si$ in the class:
\begin{enumerate}\itemsep=0pt
\item[$(a)$] The equalities $\nabla^\si T^\si=0$, $s.(T^\si)_\si=(T^\si)_\si$, $\nabla^\si R^\si=0$ and $s.(R^\si)_\si=(R^\si)_\si$ hold for the torsion and the curvature of the Weyl connection $\nabla^\si$.
\item[$(b)$] The Weyl connection $\nabla^\si$ is $\un S$-invariant.
\item[$(c)$] All $($local$)$ automorphisms of the parabolic geometry cover affine transformations of~$\nabla^\si$.
\item[$(d)$] All $($local$)$ diffeomorphisms $s_x^\si$ are affine transformations of~$\nabla^\si$.
\end{enumerate}
Two Weyl structures $\si$ and $\si\exp(\U)_\si$ from the class differ by a $G_0$-equivariant function $(\U)_\si\colon \ba_0\to \fp_+^s(1)$ which is invariant with respect to all $($local$)$ automorphisms of the parabolic geometry and is provided by an invariant element of~$\fp_+^s(1)$.
\item[$3.$] The pseudo-group generated by all local $s$-symmetries is transitive on $M$ and its connected component of identity is generated by the flows of the Lie algebra $\fl$, which is the vector subspace of $\fg_-\oplus \fg_0$, generated by $\fg_-$ by the bracket $(T^\si+R^\si)_\si$ on $\wedge^2 \fg_-^*\otimes \fl$ and the natural bracket on the rest of $\fl$.
\item[$4.$] The equalities \begin{gather*}\un S(x)=\un s_{\si(u_0)}=s_x^\si\end{gather*} hold for any Weyl structure $\si$ from~$(2)$. In particular,
\begin{itemize}\itemsep=0pt
\item the maps $\un S(x)$ can be extended to a larger neighbourhood of $x$ as long as the corresponding geodesic transformations of $\nabla^\si$ are defined,
\item for each eigenvalue $a$, the distribution $TM^s(a)$ is preserved by all $($local$)$ automorphisms of the parabolic geometry.
\end{itemize}
\item[$5.$] The distribution $TM^s(1)$ is integrable and for each $x \in M$, the leaf $\mathcal{F}_x$ of the foliation $\mathcal{F}$ of $TM^s(1)$ through $x$ is a totally geodesic submanifold for arbitrary Weyl structure.

Let $n\colon U_x\to N$ be a sufficiently small local leaf space of $TM^s(1)$.
\begin{enumerate}\itemsep=0pt
\item[$(a)$] There is a unique local diffeomorphism $\un S(n(y))$ of the local leaf space $N$ at each $n(y)\in N$ such that $\un S(n(y))\circ n=n\circ \un S(x)$ holds for all $y\in U_x$, and
\item[$(b)$] for each eigenvalue $a$, $T_yn(T_yM^s(a))$ is the $a$-eigenspace of $T_{n(y)}\un S(n(x))$ in $T_{n(y)}N$ for all $y\in U_x$.
\end{enumerate}
\end{enumerate}
\end{Theorem}
\begin{proof}
The proof is similar to the proof of Theorem \ref{6.2}. However, we need a dif\/ferent method to prove the local homogeneity in Claim (3), because the existence of some $S$-invariant Weyl structure does not follow from Proposition \ref{almost-ws} anymore. Therefore we need an additional assumption on the system $S$ on $U$ in order to apply the following lemma. Nevertheless, the fact from Proposition \ref{almost-ws} that $S(x)=s_{\si(u_0)}$ holds for any almost $S$-invariant Weyl structure $\si$ implies that the system $S$ is smooth on $U$.

\begin{Lemma}\label{infaut}
Suppose the smooth system of $($local$)$ $s$-symmetries $S$ on $M$ satisfies that the maps $\un S(x)(y)$ and $\un S(x)\circ \un S(y)^{-1}(z)$ are defined on neighbourhoods of diagonals in $M\times M$ and $M\times M\times M$.
\begin{itemize}\itemsep=0pt
\item If $c(t)$ is a curve in $M$ such that $c(0)=x$ and $\xi:=\frac{d}{dt}|_{t=0}c(t)$,
then the vector field
\begin{gather*}L_\xi(y):=\frac{d}{dt}\Big|_{t=0} \un S(c(t)) \circ \un S(x)^{-1}(y)\end{gather*}
 is defined for $y$ in some neighbourhood of $x$ in $M$.
\item Then $L_\xi(y)$ is an infinitesimal automorphism of the parabolic geometry.
\item If $\xi$ is contained in the $a$-eigenspace of $T_x\un S(x)$, then $L_\xi(x)=(1-a)\xi$.
\item The map $\xi\mapsto L_\xi$ for $\xi \in T_xM$ is a linear map onto the Lie algebra of local infinitesimal automorphisms of the parabolic geometry. Its kernel consists of the $1$-eigenspace of $T_x\un S(x)$ in $T_xM$, and it is injective on the sum of the remaining eigenspaces in $T_xM$.
\end{itemize}
\end{Lemma}
\begin{proof}[Proof of Lemma \ref{infaut}]
Since $S(c(0)) \circ S(x)^{-1}=\id_\ba$, there is a natural lift of $L_\xi(y)$ onto the $P$-invariant vector f\/ield $\frac{d}{dt}|_{t=0}S(c(t))\circ S(x)^{-1}(u)$ for $u\in \ba$ in the f\/iber over $y$. Since $S(c(t)) \circ S(x)^{-1}$ is an automorphism, the vector f\/ield is $P$-invariant and $\frac{d}{dt}|_{t=0}(S(c(t)) \circ S(x)^{-1})^*\om=0$. Thus $L_\xi(y)$ is an inf\/initesimal automorphism.

Since $\un S(c(t))(c(t))=c(t)$, we conclude that $L_\xi(x)+(\un S(x))_*(\xi)=\xi$. Thus $L_\xi(x)=\xi-(\un S(x))_*(\xi)$ and the claim follows due to the linearity of $T_x\un S(x)$.
\end{proof}

Let us continue in the proof of Theorem \ref{6.4}. Since the map $\xi \mapsto L_\xi$ from Lemma~\ref{infaut} is injective on the bracket generating distribution given by $\fg_{-1}$ due to the assumption \smash{$\fg_{-1}^s(1)=0$}, the local homogeneity follows from the regularity of the parabolic geometry. This implies Claim~(1). Then Claim~(4) follows again from Claim~(2).

Since we are on a (locally) homogeneous (locally) $s$-symmetric parabolic geometry, the parabolic geometry can be described as in Theorem~\ref{5.1}. It follows from \cite[Theorem~1.3]{GZ3} that there is a $K$-invariant Weyl connection $\nabla$ on the $K$-homogeneous parabolic geometry described Theorem~\ref{5.1} such that all local automorphisms of the parabolic geometry are af\/f\/ine transformations of $\nabla$. Therefore it follows from the last claim of Theorem~\ref{5.1} that the pullback of~$\nabla$ to~$M$ does not depend on the local isomorphism with the $K$-homogeneous parabolic geo\-met\-ry. Therefore we obtain a Weyl structure~$\si$ that satisf\/ies~(2c), which implies the remaining parts~(2a),~(2b) and~(2d). It is clear that the $K$-invariant Weyl connection $\nabla$ from \cite[Theorem~1.3]{GZ3} is not unique and the dif\/ference between two such Weyl structures is the claimed $\U$ provided by a $K$-invariant element of~$\fp_+^s(1)$.

Proposition~\ref{geodesic-s-symm} implies that the Weyl structure $\si$ satisf\/ies (2b) if and only if it satisf\/ies~(2d). Again, results in \cite[Theorem~1.3]{GZ3} imply that (2b) implies (2c) and the same arguments as in the proof of Theorem~\ref{6.2} show that~(2b) implies~(2a) and~(2a) implies~(2d).

To prove Claim (5), we use the fact that $s.(T^\si(x))_\si=(T^\si(x))_\si$ holds for the torsion of the $\un S$-invariant Weyl connection~$\na^\si$. Thus $TM^s(1)$ is involutive, because each (almost) $\un S$-invariant Weyl connection $\nabla^\si$ preserves~$TM^s(1)$. Moreover, the formula for the dif\/ference between $\nabla^\si$ and arbitrary Weyl connection implies that the dif\/ference in the parallel transport is an element of~$TM^s(1)$ at each point of~$\mathcal{F}_x$. Thus~$\mathcal{F}_x$ is a totally geodesic submanifold for any Weyl connection.

We know that $\un S(x)=s^{\si}_x$ and this implies $\un S(x)|_{\mathcal{F}_x}=s^{\si}_x|_{\mathcal{F}_x}=\id_{\mathcal{F}_x}$. If $v=\Fl_1^{\om^{-1}(X)}(u)$ for $X\in \fg_{-}^s(1)$, then $S(x)v=vs$ holds and $y=p\circ \Fl_1^{\om^{-1}(X)}(u)\in \mathcal{F}_x$, because $\mathcal{F}_x$ is a totally geodesic submanifold. Thus $\un S(x)$ is covered by the $s$-symmetry at~$y$ and $S(x)=S(y)$ holds in some neighbourhood of $x$ due to the uniqueness of $s$-symmetries. Consequently, Claim (5a) holds on a suf\/f\/iciently small local leaf space and Claim~(5b) is a clear consequence of Claim~(4).
\end{proof}

\subsection[Parabolic geometries with $\fg_{-1}^s(1)+\fp$ in a distinguished parabolic subalgebra]{Parabolic geometries with $\boldsymbol{\fg_{-1}^s(1)+\fp}$ \\ in a distinguished parabolic subalgebra}\label{sec5.4}

There are triples $(\fg,\fp,\mu)$ that are prolongation rigid outside of the $1$-eigenspace of $s$ that admit a $1$-eigenspace in $\fg_-$ for some $s$ such that $\fg_{-1}^s(1)+\fp\subset \fv\subset \fg_{-}^s(1)+\fp$ holds for some parabolic subalgebra $\fv$ of $\fg$ such that the harmonic curvature vanishes on insertions of elements of $\fv/\fp$ at all points of~$M$. These are listed in Tables~\ref{ahs7}, \ref{ahs8} and \ref{ahs9}, due to \cite[Propositions~6.2 and~A.2]{GZ3}.

\begin{table}[th!]\centering\caption{Theorem \ref{6.5}, part with $|\Si|=2$.}\label{ahs7}\vspace{1mm}

\begin{tabular}{|c|c|c|c|c|c|}
\hline
$\fg$ & $\Sigma$ & $j_{i_1}$& $j_{i_2}$& $\mu$&PR \\
\hline
\hline
$\mathfrak{sl}(4,\C)$ & $\{1,2\}$ &&$j_2^2$ & $(\alpha_2,\alpha_3)$ & \\
\hline
$\mathfrak{sl}(4,\C)$ & $\{1,3\}$ &&$j_1^2$ &$(\alpha_1,\alpha_2)$& \\
\hline
$\mathfrak{su}(2,2)$&$\{1,2\}$&& $2r_1$&$(\alpha_1,\alpha_{3})$&$r_1=0$\\
\hline
$\mathfrak{sl}(n+1,\C)$ & $\{1,3\}$ &&$j_1^2$ &$(\alpha_1,\alpha_2)$& \\
\hline
$\mathfrak{sl}(n+1,\C)$ & $\{2,3\}$ &&$j_2^2$ & $(\alpha_2,\alpha_1)$ & \\
\hline
$\mathfrak{sl}(n+1,\C)$ & $\{1,n\}$ &&$j_1^2$ &$(\alpha_1,\alpha_2)$& \\
\hline
$\mathfrak{sl}(n+1,\C)$ & $\{2,n\}$ &&$j_2^2$ & $(\alpha_2,\alpha_1)$ & \\
\hline
$\mathfrak{su}(n,n)$&$\{1,n\}$&& $2r_1$&$(\alpha_1,\alpha_{2n-1})$&$r_1=0$\\
\hline
$\mathfrak{so}(7,\C)$&$\{ 1,3\}$&$j_3^3$&& $(\alpha_3,\alpha_2)$& \\
\hline
$\mathfrak{so}(8,\C)$&$\{ 1,3\}$& & $j_1^2$& $(\alpha_1,\alpha_2)$& \\
\hline
$\mathfrak{sp}(4,\mathbb{C})$& $\{1,2\}$ && $2r_1$& $(\alpha_{1},\alpha_{1'})$ & $r_1=0$ \\
\hline
$\mathfrak{sp}(6,\{\R,\C\})$ & $\{1,2\}$ & $j_2^{-2}$ & & $(\alpha_2,\alpha_1)$ & $j_2=-1$ \\
\hline
$\mathfrak{sp}(6,\{\R,\C\})$ & $\{1,2\}$ & &$j_1^2$ & $(\alpha_1,\alpha_2)$&\\
\hline
$\mathfrak{sp}(6,\{\R,\C\})$ & $\{1,3\}$& & $j_1^2$& $(\alpha_1,\alpha_2)$ & \\
\hline
$\mathfrak{sp}(6,\mathbb{C})$& $\{1,3\}$ && $2r_1$& $(\alpha_{1},\alpha_{1'})$ & $r_1=0$ \\
\hline
$\mathfrak{sp}(6,\{\R,\C\})$& $\{2,3\}$ & & $j_2^2$& $(\alpha_2,\alpha_1)$ & \\
\hline
$\mathfrak{sp}(2n,\{\R,\C\})$ & $\{1,2\}$ & $j_2^{-2}$ & & $(\alpha_2,\alpha_1)$ & $j_2=-1$ \\
\hline
$\mathfrak{sp}(2n,\{\R,\C\})$ & $\{1,n\}$& & $j_1^2$& $(\alpha_1,\alpha_2)$ &$j_1=-1$ \\
\hline
$\mathfrak{sp}({n \over 2},{n \over 2}),\mathfrak{sp}(2n,\{\R,\C\})$& $\{2,n\}$ & & $j_2^2$& $(\alpha_2,\alpha_1)$ & $j_2=-1$ \\
\hline
$\mathfrak{sp}(2n,\{\R,\C\})$ & $\{1,2\}$ & &$j_1^2$ & $(\alpha_1,\alpha_2)$&\\
\hline
$\mathfrak{sp}(2n,\mathbb{C})$& $\{1,n\}$ && $2r_1$& $(\alpha_{1},\alpha_{1'})$ & $r_1=0$ \\
\hline
\end{tabular}
\end{table}

\begin{table}[th!]\centering\caption{Theorem \ref{6.5}, part with $|\Si|=3$.}\label{ahs8}\vspace{1mm}

\begin{tabular}{|c|c|c|c|c|c|c|}
\hline
$\fg$ & $\Sigma$ & $j_{i_1}$& $j_{i_2}$& $j_{i_3}$&$\mu$&PR \\
\hline
\hline
$\mathfrak{sl}(4,\mathbb{C})$ &$\{1,2,3\}$& & & $j_1j_2^2 $ &$(\alpha_2,\alpha_1)$& $j_1=1$ \\
\hline
$\mathfrak{sl}(n+1,\mathbb{C})$ &$\{1,2,3\}$& & & $j_1j_2^2 $ &$(\alpha_2,\alpha_1)$& $j_1=1$ \\
\hline
$\mathfrak{sl}(n+1,\mathbb{C})$ &$\{1,2,n\}$& & & $j_1j_2^2 $ &$(\alpha_2,\alpha_1)$& $j_1=1$ \\
\hline
$\mathfrak{sl}(n+1,\{\mathbb{R, C}\})$&$\{1,p,n\}$, $p>2$ && $j_1j_n$& &$(\alpha_1,\alpha_n)$&$j_1=j_n^{-1}$ \\
\hline
$\mathfrak{so}(8,\mathbb{C})$& $\{1,2,4\}$&&& $j_1^2 $&$(\alpha_1,\alpha_2)$& $j_2=1$\\
\hline
$\mathfrak{sp}(6,\{\mathbb{R, C}\})$&$\{1,2,3\}$&&& $j_1j_2^2 $&$(\alpha_2,\alpha_1)$ &$j_1=1$\\
\hline
$\mathfrak{sp}(2n,\{\mathbb{R, C}\})$&$\{1,2,p\}$, $p<n$&& & $\sqrt{j_1j_2^2} $&$(\alpha_2,\alpha_1)$&$j_1=1$, $j_2=-1$, $j_p=1$ \\
\hline
$\mathfrak{sp}(2n,\{\mathbb{R, C}\})$&$\{1,2,n\}$&&& $j_1j_2^2 $&$(\alpha_2,\alpha_1)$ &$j_1=1$, $j_2=-1$\\
\hline
\end{tabular}
\end{table}

\begin{table}[th!]\centering\caption{Theorem \ref{6.5}, part with $|\Si|=4$.}\label{ahs9}\vspace{1mm}

\begin{tabular}{|c|c|c|c|c|}
\hline
$\fg$ & $\Sigma$ & eigenvalues& $\mu$& PR \\
\hline
\hline
$\mathfrak{sl}(n+1,\mathbb{C})$ & $\{1,2,3,q\}$, $q<n$&$j_q=j_1j_2^2j_p^{-1}$&$(\alpha_2,\alpha_1)$& $j_1=1$, $j_p=j_2^2$\\
\hline
$\mathfrak{sl}(n+1,\mathbb{C})$ & $\{1,2,p,n\}$, $3<p$&$j_n=j_1j_2^2j_p^{-1}$&$(\alpha_2,\alpha_1)$& $j_1=1$, $j_p=1$\\
\hline
\end{tabular}
\end{table}

\begin{Example}
Let us focus on Lagrangean complex contact geometries, i.e., $\fg=\mathfrak{sl}(n+1,\mathbb{C})$ and $\Sigma=\{1,n\}$. If we consider the triple $(\mathfrak{sl}(n+1,\mathbb{C}),\fp_{\{1,n\}},\mu_{(\alpha_1,\alpha_2)})$ from Table~\ref{ahs7} that is prolongation rigid outside of the $1$-eigenspace of~$s$, then dif\/ferent situations arise depending on the choice of $s\in Z(G_0)$. If $j_1=-1$, then $\fv=\fg_{-}^s(1)+\fp$ is a parabolic subalgebra satisfying the assumptions of Theorem~\ref{6.3}. If $j_1=\sqrt[3]{1}$, then $\fg_{-}^s(1)=\fg_{-2}$ and we need the assumptions of Theorem~\ref{6.4} to state the results. We can apply Theorem~\ref{6.2} for the other values~$j_1$.
\end{Example}

In general, $\fv$ can be a proper subspace of $\fg_{-}^s(1)+\fp$ and we can (locally) apply the general result for parabolic geometries from~\cite{cap-cor} to obtain the following theorem.

\begin{Theorem}\label{6.5}
Assume $(\fg,\fp,\mu)$ is prolongation rigid outside of the $1$-eigenspace of $s$ for $s\in Z(G_0)$ such that $\fv$ is a maximal parabolic subalgebra of $\fg$ such that $\fg_{-1}^s(1)+\fp\subset \fv\subset \fg_{-}^s(1)+\fp$ and $\fv/\fp$ inserts trivially into the harmonic curvature of the (locally) $s$-symmetric parabolic geometry $(\ba\to M,\om)$ of type $(G,P)$. Assume the open subset $U$ of $M$ containing all points at which $\kappa_H$ has a non-zero component in $\mu$ is non-trivial, and the maps $\un S(x)(y)$ and $\un S(x)\circ \un S(y)^{-1}(z)$ are defined on neighbourhoods of diagonals in $U\times U$ and $U\times U\times U$ for the unique system $S$ of $($local$)$ $s$-symmetries on $U$. Then:
\begin{enumerate}\itemsep=0pt
\item[$1.$] The set $U$ is an open dense subset of $M$ and there is a unique smooth system of $($local$)$ $s$-symmetries $S$ on $M$.
\item[$2.$] There are
\begin{itemize}\itemsep=0pt
\item a parabolic subgroup $Q$ of $G$ with the Lie algebra $\fv$ such that $P\subset Q$,
\item a neighbourhood $U_x$ of each $x\in M$ with a local leaf space $n\colon U_x\to N$ for the foliation given by the integrable distribution $Tp\circ \om^{-1}(\fv)$, and
\item a $($locally$)$ $s$-symmetric parabolic geometry $(\ba'\to N,\om')$ of type $(G,Q)$ satisfying the assumptions of Theorem~{\rm \ref{6.4}},
\end{itemize}
such that $(\ba|_{U_x}\to U_x,\om|_{U_x})$ is isomorphic to an open subset of $(\ba'\to \ba'/P,\om')$ for each~$x$. In particular, there is a unique $s$-symmetry $\un S(n(y))$ of $(\ba'\to N,\om')$ at each $n(y)\in N$ such that $n\circ \un S(y)=\un S(n(y))\circ n$ holds for all $y\in U_x$ in the fiber over $n(y)$.
\item[$3.$] The connected component of identity of the pseudo-group generated by all local $s$-sym\-met\-ries is generated by the flows of the Lie algebra $\fl$, which is the vector subspace of $\fv^{\rm op}$, generated by $\fv^{\rm op}_+$ by the bracket $(T^\si(n(x))+R^\si(n(x)))_\si$
on $\wedge^2 (\fv^{\rm op}_+)^*\otimes \fl$ and the natural bracket on the rest of $\fl$ for arbitrary $S$-invariant Weyl structure $\si$ on $(\ba'\to N,\om')$.

The pseudo-group generated by all local $s$-symmetries is locally transitive at~$x$ if and only if $\fv/\fp\subset \fl/(\fl\cap \fp)$, i.e., if and only if $(R^\si(n(x)))_\si$ spans the whole $\fv/\fp$.
\item[$4.$] There is a class of almost $S$-invariant Weyl structures on $U_x$ given by reductions of the images in~$\ba'$ of the $S$-invariant Weyl structures on~$N$ $($that exist due to Theorem~{\rm \ref{6.4})} to $\exp(\fg_-^s(1))\rtimes G_0\subset Q_0$. A~reduction corresponds to an $S$-invariant Weyl structure on~$U_x$ if and only if it is a~holonomy reduction.
\item[$5.$] We get that
\begin{itemize}\itemsep=0pt
\item the maps $\un S(x)$ can be extended to a larger neighbourhood of $x$ as long as the corresponding geodesic transformations of~$\nabla^\si$ on~$N$ are defined,
\item the space $TM^s(1)$ is integrable, it contains the vertical space of the local leaf space $n\colon U_x\to N$, and $Tn(T_xM^s(a))$ is the $a$-eigenspace of $T_{n(x)}\un S(n(x))$ in~$T_{n(x)}N$,
\item for each eigenvalue $a$, the distribution $TM^s(a)$ is preserved by all $($local$)$ automorphisms of the parabolic geometry, and
\item all almost $\un S$-invariant Weyl connections from Claim~$(4)$ restrict to the same partial linear connection on~$TM$ corresponding to the distribution $\ba_0\times_{G_0}\fv^{\rm op}_+$, which is preserved by~$\un S(x)$ for all $x\in M$.
\end{itemize}
\end{enumerate}
\end{Theorem}
\begin{proof}
Claim (1) is a direct consequence of Claims~(2) and~(3). Claim~(2) follows from~\cite{cap-cor} and the fact that $(\fv^{\rm op}_1)^s(1)=0$ holds. Then Claim~(3) is a clear consequence of Theorem~\ref{6.4}. Claim~(4) follows from the comparison of images in $\ba'$ of the $S$-invariant Weyl structure on~$N$ and the almost $S$-invariant Weyl structures on $U_x$, because they intersect precisely in the reduction to $\exp(\fg_-^s(1))\rtimes G_0\subset Q_0$. Claim~(5) is a consequence of Claim~(5) of Theorem~\ref{6.4} and the properties of Weyl structures from Claim (4).
\end{proof}

\subsection[Parabolic geometries with $\fg_{-1}^s(1)$ that inserts non-trivially into the harmonic curvature]{Parabolic geometries with $\boldsymbol{\fg_{-1}^s(1)}$ that inserts non-trivially\\ into the harmonic curvature}

There are also some remaining parabolic geometries, which can have a part of $\fg_{-1}^s(1)$ that inserts non-trivially into the harmonic curvature. These are contained in Tables~\ref{ahs10},~\ref{ahs11} and~\ref{ahs12}.

\begin{table}[th!]\centering\caption{Theorem \ref{6.6}, part with $|\Si|=2$.}\label{ahs10}\vspace{1mm}

\begin{tabular}{|c|c|c|c|c|}
\hline
$\fg$ & $\Sigma$ & $j_{i_1}$& $j_{i_2}$& $\mu$ \\
\hline
\hline
$\mathfrak{sl}(3,\mathbb{C})$ & $\{1,2\}$ &$2\phi_2$&& $(\alpha_{1},\alpha_{2'})$ \\
\hline
$\mathfrak{sl}(4,\{\R,\C\})$ & $\{1,2\}$ & $\sqrt{1}$ & & $(\alpha_1,\alpha_2)$ \\
\hline
$\mathfrak{sl}(4,\mathbb{C})$ & $\{1,2\}$ &$2\phi_2$&& $(\alpha_{1},\alpha_{2'})$ \\
\hline
$\mathfrak{sl}(4,\mathbb{C})$ & $\{1,3\}$ &$2\phi_3$&& $(\alpha_{1},\alpha_{3'})$ \\
\hline
$\mathfrak{su}(2,2)$&$\{1,2\}$&$r_1$& $\sqrt{1} $& $(\alpha_2,\alpha_{1})$\\
\hline
$\mathfrak{su}(2,2)$&$\{1,2\}$&$ \sqrt[3]{1}$&&$(\alpha_1,\alpha_{2})$ \\
\hline
$\mathfrak{sl}(n+1,\{\R,\C\})$ & $\{1,2\}$ & $\sqrt{1}$ & & $(\alpha_1,\alpha_2)$ \\
\hline
$\mathfrak{sl}(n+1,\{\mathbb{R, C}\})$& $\{1,p\}$, $2<p<n$& $1$& & $(\alpha_1,\alpha_p)$\\
\hline
$\mathfrak{sl}(n+1,\mathbb{C})$ & $\{1,p\}$ &$2\phi_p$&& $(\alpha_{1},\alpha_{p'})$ \\
\hline
$\mathfrak{su}(q,n-q+1)$ &$\{1,2\}$&&$-\frac23\phi_1$&$(\alpha_2,\alpha_1)$\\
\hline
$\mathfrak{so}(3,4), \mathfrak{so}(7,\C)$&$\{ 2,3\}$&&$\sqrt[3]{1}$& $(\alpha_3,\alpha_2)$ \\
\hline
$\mathfrak{so}(2,5)$, $\mathfrak{so}(3,4)$, $\mathfrak{so}(7,\mathbb{C})$& && & \\
$\mathfrak{so}(2,6)$, $\mathfrak{so}(3,5)$,& $\{1,2\}$ &&$1$ & $(\alpha_2,\alpha_1)$\\
$\mathfrak{so}(4,4)$, $\mathfrak{so}(7,\mathbb{C})$& & && \\
\hline
$\mathfrak{so}(3,5)$&$\{2,3\}$& &$\sqrt[3]{1} $&$(\alpha_3,\alpha_2)$\\
\hline
$\mathfrak{so}(q,n-q)$, $\mathfrak{so}(n,\mathbb{C})$& $\{1,2\}$ &&$1$ & $(\alpha_2,\alpha_1)$\\
\hline
$\mathfrak{sp}(4,\{\mathbb{R, C}\})$& $\{1,2\}$ &$\sqrt[3]{1}$&& $(\alpha_1,\alpha_2)$ \\
\hline
$\mathfrak{sp}(4,\mathbb{C})$ & $\{1,2\}$&$2\phi_2$& & $(\alpha_{1},\alpha_{2'})$ \\
\hline
$\mathfrak{sp}(4,\mathbb{C})$ & $\{1,2\}$&$-\frac25\phi_2$&& $(\alpha_{1'},\alpha_{2'})$ \\
\hline
$\mathfrak{sp}(6,\{\mathbb{R, C}\})$& $\{1,3\}$ &$1$&& $(\alpha_1,\alpha_3)$ \\
\hline
$\mathfrak{sp}(6,\{\mathbb{R, C}\})$& $\{2,3\}$&$1$&& $(\alpha_{2},\alpha_3)$ \\
\hline
$\mathfrak{sp}(6,\mathbb{C})$ & $\{1,3\}$&$2\phi_3$& & $(\alpha_{1},\alpha_{3'})$ \\
\hline
$\mathfrak{sp}(6,\mathbb{C})$ & $\{2,3\}$&$-\frac25\phi_3$&& $(\alpha_{2'},\alpha_{3'})$ \\
\hline
$\mathfrak{sp}(2n,\{\mathbb{R, C}\})$& $\{1,n\}$ &$1$&& $(\alpha_1,\alpha_n)$ \\
\hline
$\mathfrak{sp}(2n,\{\mathbb{R, C}\})$& $\{n-1,n\}$&$1$&& $(\alpha_{n-1},\alpha_n)$ \\
\hline
$\mathfrak{sp}(2n,\mathbb{C})$ & $\{1,n\}$&$2\phi_n$& & $(\alpha_{1},\alpha_{n'})$ \\
\hline
$\mathfrak{sp}(2n,\mathbb{C})$ & $\{n-1,n\}$&$-\frac25\phi_n$&& $(\alpha_{n-1'},\alpha_{n'})$ \\
\hline
$\mathfrak{g}_2(\{2,\C\})$ & $\{1,2\}$ & $\sqrt[4]{1}$ && $(\alpha_1,\alpha_2)$ \\
\hline
\end{tabular}
\end{table}

\begin{table}[th!]\centering\caption{Theorem \ref{6.6}, part with $|\Si|=3$.}\label{ahs11}\vspace{1mm}

\begin{tabular}{|c|c|c|c|c|c|c|}
\hline
$\fg$ & $\Sigma$ & $j_{i_1}$& $j_{i_2}$& $j_{i_3}$&$\mu$&PR \\
\hline
\hline
$\mathfrak{sl}(4,\{\mathbb{R, C}\})$ &$\{1,2,3\}$ &&&$j_1^2$&$(\alpha_1,\alpha_2)$ & \\
\hline
$\mathfrak{sl}(n+1,\{\mathbb{R, C}\})$ &$\{1,2,3\}$ &&&$j_1^2$&$(\alpha_1,\alpha_2)$ & \\
\hline
$\mathfrak{sl}(n+1,\{\mathbb{R, C}\})$ &$\{1,2,p\}$, $3<p<n$ &&&$j_1^2$&$(\alpha_1,\alpha_2)$ & $j_1=\sqrt{1}$ \\
\hline
$\mathfrak{sl}(n+1,\{\mathbb{R, C}\})$ &$\{1,2,n\}$ &&&$j_1^2$&$(\alpha_1,\alpha_2)$ &\\
\hline
$\mathfrak{sl}(n+1,\{\mathbb{R, C}\})$&$\{1,2,n\}$ && $j_1j_n$& &$(\alpha_1,\alpha_n)$& \\
\hline
$\mathfrak{so}(3,4)$, $\mathfrak{so}(7,\mathbb{C})$&$\{1,2,3\}$& $j_3^3 $&&&$(\alpha_3,\alpha_2)$& \\
\hline
\end{tabular}
\end{table}

\begin{table}[th!]\centering\caption{Theorem \ref{6.6}, part with $|\Si|=4$.}\label{ahs12}\vspace{1mm}

\begin{tabular}{|c|c|c|c|c|c|c|c|c|}
\hline
$\fg$ & $\Sigma$ & $j_{i_1}$& $j_{i_2}$& $j_{i_3}$& $j_{i_4}$& $\mu$& PR \\
\hline
\hline
$\mathfrak{sl}(n+1,\{\mathbb{R, C}\})$ & $\{1,2,3,n\}$&&&&$j_1j_2^2j_p^{-1}$&$(\alpha_2,\alpha_1)$&$j_1=1$\\
\hline
\end{tabular}
\end{table}

\begin{Example}
Let us continue in the discussion of generalized path geometries from Examp\-le~\ref{expg}. The case when the harmonic curvature $\kappa_{(\alpha_2,\alpha_1)}$ vanishes and the harmonic torsion $\kappa_{(\alpha_1,\alpha_2)}$ does not vanish can be found in Table~\ref{ahs10}. There are several possible situations depending on the eigenvalues of $s\in Z(G_0)$.

If $j_1=1$, then we are precisely in the situation which is not covered by any of the previous theorems and we can apply only the results of Propositions~\ref{almost-ws} and~\ref{6.6}.

If $j_1=-1$ and $j_2=1$, then we can apply Theorem~\ref{6.3} and we are in the situation of a~generalized path geometry on the projectivized cotangent space of an af\/f\/ine locally symmetric space.

If $j_1=-1$ and $j_2=-1$, then $\fg_-^s(1)=\fg_{-2}$ and we need the assumptions of Theorem~\ref{6.4} to show that we are in the situation of a generalized path geometry on a (locally) homogeneous $(n-1)$-dimensional f\/iber bundle over an af\/f\/ine locally symmetric space.

Finally, if $j_1=-1$ and $j_2\neq \sqrt{1}$, then we can apply Theorem~\ref{6.2}.
\end{Example}

The properties of these geometries are as follows.

\begin{Proposition}\label{6.6}
Assume $(\fg,\fp,\mu)$ is prolongation rigid outside of the $1$-eigenspace of $s$ for some $s\in Z(G_0)$. Assume the harmonic curvature $\kappa_H$ of a $($locally$)$ $s$-symmetric parabolic geo\-met\-ry $(\ba\to M,\om)$ of type $(G,P)$ has a non-zero component in $\mu$ at all $x\in M$ and $S$ is the unique system of $($local$)$ $s$-symmetries on~$M$. Then the distribution~$TM^s(1)$ is integrable and for each $x \in M$, the leaf~$\mathcal{F}_x$ of the foliation $\mathcal{F}$ of $TM^s(1)$ through~$x$ is a~totally geodesic submanifold for arbitrary Weyl connection.

Let $n\colon U_x\to N$ be a sufficiently small local leaf space of $TM^s(1)$.
\begin{itemize}\itemsep=0pt
\item There is a unique local diffeomorphism $\un S(n(y))$ of the local leaf space $N$ at each $n(y)\in N$ such that $\un S(n(y))\circ n=n\circ \un S(x)$ holds for all $y\in U_x$, and
\item for each eigenvalue $a$, $T_yn(T_yM^s(a))$ is the $a$-eigenspace of $T_{n(y)}\un S(n(y))$ for all $y\in U_x$.
\end{itemize}
\end{Proposition}
\begin{proof}
The proof is analogous to the proof of Claim (6) of Theorem~\ref{6.4}, but, instead of an $S$-invariant Weyl structure~$\si$, we need to consider some almost $S$-invariant Weyl structure invariant at $x$ from Claim (3) of Proposition~\ref{almost-ws} for each $x\in M$.
\end{proof}

\subsection{Parabolic geometries that do not admit non-f\/lat examples}

There are triples $(\fg,\fp,\mu)$ that are prolongation rigid outside of the $1$-eigenspace of~$s$ for some $s\in Z(G_0)$, but they admit only f\/lat (locally) $s$-symmetric parabolic geometries due to the structure of the harmonic curvature and \cite[Lemma~2.2]{GZ2}. These are contained in Table~\ref{ahs13}.

\begin{table}[th!]\centering \caption{Flat geometries.}\label{ahs13}\vspace{1mm}

\begin{tabular}{|c|c|c|c|c|c|}
\hline
$\fg$ & $\Sigma$ & $j_{i_1}$& $j_{i_2}$& $\mu$ \\
\hline
\hline
$\mathfrak{sl}(3,\mathbb{C})$ & $\{1\}$&$\sqrt[3]{1}$&$1$&$(\alpha_{1},\alpha_2)$\\
\hline
$\mathfrak{sl}(3,\{\mathbb{R},\mathbb{C}\})$ & $\{1,2\}$&$j_1^4\neq1, j_1^5\neq 1$&$j_1^3$&$(\alpha_{1},\alpha_2)$\\
\hline
$\mathfrak{so}(5,\mathbb{C})$ & $\{1\}$&$\sqrt[3]{1}$&$1$&$(\alpha_{1},\alpha_2)$\\
\hline
$\mathfrak{so}(2,3), \mathfrak{so}(5,\mathbb{C})$ & $\{1,2\}$&$j_1^4\neq1, j_1^5\neq 1, j_1^7\neq 1$&$j_1^3$&$(\alpha_{1},\alpha_2)$\\
\hline
\end{tabular}
\end{table}

\subsection[Parabolic geometries with more non-zero components of the harmonic curvature]{Parabolic geometries with more non-zero components \\ of the harmonic curvature}

Let us also look at the parabolic geometries that allow a harmonic curvature~$\kappa_H$ with several non-zero components $\mu_i$ such that for each~$\mu_i$ the triple $(\fg,\fp,\mu_i)$ is not prolongation rigid outside of the $1$-eigenspace of $s$. In Table~\ref{ahs14}, we present the complete classif\/ication of all triples $(\fg,\fp,\mu_i)$ that are not prolongation rigid outside of the $1$-eigenspace of~$s$ for the same $s\in Z(G_0)$, but for which $\mathfrak{a}_i$ in Proposition~\ref{multigrad} is contained in the $1$-eigenspaces of $s$ when the harmonic curvature has non-zero component in each $\mu_i$. Geometric properties of the geometries from Table~\ref{ahs14} can be deduced from the previous sections depending on the position and shape of $\fg_{-}^s(1)$ inside of~$\fg_-$.

\begin{table}[th!]\centering \caption{More non-zero components of the harmonic curvature.}\label{ahs14}\vspace{1mm}

\begin{tabular}{|c|c|c|c|c|c|c|c|}
\hline
$\fg$ & $\Sigma$ & eigenvalues & $\mu$ \\
\hline
\hline
$\mathfrak{sl}(4,\C)$ & $\{1,2\}$ & $j_1=\sqrt[4]{1}^{2}, j_2=\sqrt[4]{1}$ &$(\alpha_2,\alpha_3)$, $(\alpha_2,\alpha_1)$ \\
\hline
$\mathfrak{sl}(n+1,\mathbb{C})$ &$\{1,n-1\}$&$j_1=\sqrt[3]{1},j_{n-1}=\sqrt[3]{1}^2$& $(\alpha_1,\alpha_2)$, $(\alpha_{n-1},\alpha_n)$ \\
\hline
$\mathfrak{sl}(n+1,\mathbb{C})$ &$\{2,n-1\}$&$j_2=\sqrt[3]{1},j_{n-1}\sqrt[3]{1}^2$& $(\alpha_2,\alpha_1)$, $(\alpha_{n-1},\alpha_n)$ \\
\hline
$\mathfrak{sl}(4,\mathbb{C})$&$\{1,2,3\}$ & $j_2=\sqrt[4]{1},j_3=j_1(\sqrt[4]{1})^2$ &$(\alpha_2,\alpha_1)$, $(\alpha_2,\alpha_3)$ \\
\hline
$\mathfrak{sl}(n+1,\{\mathbb{R, C}\})$ & $\{1,2,n-1,n\}$&$j_{n-1}=j_2^{-1}, j_n=j_1j_2^3$&$(\alpha_2,\alpha_1)$, $(\alpha_{n-1},\alpha_n)$\\
\hline
\end{tabular}
\end{table}

\subsection[Remaining parabolic geometries with $\mu$ in the 1-eigenspace of $s$]{Remaining parabolic geometries with $\boldsymbol{\mu}$ in the 1-eigenspace of $\boldsymbol{s}$}

For the sake of completeness, let us remark that there are triples $(\fg,\fp,\mu)$ that are not prolongation rigid outside of the $1$-eigenspace of~$s$ for any $s$ such that $\mu$ is in the $1$-eigenspace of $s$. These are contained in Table~\ref{nhs}.

\begin{table}[th!]\centering \caption{Remaining parabolic geometries with $\mu$ in the $1$-eigenspace of $s$.}\label{nhs}\vspace{1mm}

\begin{tabular}{|c|c|c|c|c|c|}
\hline
$\fg$ & $\Sigma$ & $j_{i_1}$& $j_{i_2}$& $\mu$ \\
\hline
\hline
$\mathfrak{sl}(n+1,\{\mathbb{R, C}\})$ & $\{p,p+1\}$, $n-1>p>1$&&$1$&$(\alpha_{p+1},\alpha_p)$\\
\hline
$\mathfrak{so}(q,n-q), \mathfrak{so}(n,\mathbb{C})$ & $\{2,3\}$&&$1$&$(\alpha_{3},\alpha_2)$\\
\hline
\end{tabular}
\end{table}

\appendix

\section[A construction of locally homogeneous locally $s$-symmetric parabolic geometries]{A construction of locally homogeneous locally $\boldsymbol{s}$-symmetric\\ parabolic geometries}\label{sec5}

It is proved in \cite[Section 2]{GZ2} how to algebraically construct and classify all homogeneous $s$-symmetric parabolic geometries. Part of the classif\/ication is done in \cite{G2,G3} using the classif\/ication of semisimple symmetric spaces. There is the result from \cite[Section~1.3]{disertace} and \cite[Lemma~2.2]{GZ2} stating that for the construction and the classif\/ication of locally homogeneous locally $s$-symmetric parabolic geometries, it is suf\/f\/icient to f\/ind the following data:
\begin{itemize}\itemsep=0pt
\item an extension $(\alpha,i)$ of the Klein geometry $(K,H)$ to $(G,P)$ such that the action of $s$ preserves $\alpha(\fk)\subset \fg$, and $s$ acts trivially on the tensor $[\cdot ,\cdot ]-\alpha([\alpha^{-1}(\cdot ),\alpha^{-1}(\cdot )])$ in $\wedge^2 \fg/\fp^*\otimes \fg$, and
\item the subset $\mathcal{A}$ of $P$ consisting of elements $g\in P$, which act as local automorphisms on the parabolic geometry $(K\times_{i}P\to K/H,\om_\alpha)$ of type $(G,P)$ given by the extension~$(\alpha,i)$.
\end{itemize}

If $U$ and $V$ are open subsets of $K/H$ such that there are $k\in K$, $g\in\mathcal{A}$ and a maximal open subset $W$ of $U$ such that $kg(W)\subset V$, then we can glue $K\times_{i}P|_U\to U$ with $K\times_{i}P|_V\to V$ by identifying $w\in W\subset U$ with $kg(w)\in V$, and glue the Cartan connection $\om_\alpha|_U$ with the pullback connection $(kg)^*\om_\alpha|_V=\om_\alpha|_{(kg)^*(V)}$. Of course, we can without loss of generality assume that~$U$,~$V$ and~$W$ are simply connected, because we can always choose coverings of our manifolds by open sets satisfying this condition. Therefore, we can also assume that the automorphism~$k$ is given by the f\/low of a local inf\/initesimal automorphism of $(K\times_{i}P\to K/H,\om_\alpha)$. Then we obtain the following result as a consequence of the construction in \cite[Section~3]{GZ3} and \cite[Section~1.3]{disertace}.

\begin{Theorem} \label{5.1}
Let $(\ba\to M,\om)$ be a locally homogeneous locally $s$-symmetric parabolic geo\-metry, let $\fk$ be the Lie algebra of the local infinitesimal automorphisms and denote by $\alpha$ the inclusion of $\fk$ into $\fg$ given by $\om(u)$ at some $u\in \ba$. Then:
\begin{enumerate}\itemsep=0pt
\item[$1)$] $\Ad(s)(\fk)\subset \fk$ is an automorphism of the Lie algebra $\fk$,
\item[$2)$] there exist $($see {\rm \cite[Section 3]{GZ3}} for the explicit construction$)$
\begin{itemize}\itemsep=0pt
\item a Klein geometry $(K,H)$ such that $\fk$ is the Lie algebra of $K$,
\item an extension $(\alpha,i)$ of $(K,H)$ to $(G,P)$,
\item an open covering $U_a$ of $M$, and
\item isomorphisms $\phi_a\colon U_i\to K/H$ of parabolic geometries $(\ba|_{U_a}\to U_i,\om|_{U_b})$ and $(K\times_{i}P|_{\phi_a(U_a)}\to \phi_a(U_a), \om_\alpha|_{\phi_a(U_a)})$ of type $(G,P)$ such that $\phi_a\circ \phi_b^{-1}$ is the restriction of the left action of some element of $K$ for each $a$, $b$.
\end{itemize}
\end{enumerate}
\end{Theorem}

\subsection*{Acknowledgments}

JG supported by the Grant agency of the Czech Republic under the grant GBP201/12/G028.
The authors would like to thank the anonymous referees for their valuable comments which helped to improve the manuscript.

\pdfbookmark[1]{References}{ref}
\LastPageEnding

\end{document}